\documentclass[preprint,12pt]{elsarticle}

\usepackage[dvipdfmx,colorlinks=true]{hyperref}
\hypersetup{urlcolor=blue, citecolor=red}
\usepackage{amsmath, amssymb, bm}
\usepackage{amsthm} 
\usepackage[pagewise]{lineno}
\usepackage{mathrsfs} 
\numberwithin{equation}{section}
\usepackage{here}

\journal{arXiv}

\newtheorem{Th}{Theorem}
\newtheorem{Le}{Lemma}
\newtheorem{Co}{Corollary}
\newtheorem{Pro}{Proposition}

\newtheorem{Rem}{Remark}

\newcommand{\N}{\mathbb{N} }

\newcommand{\R}{\mathbb{R} }

\newcommand{\BS}{\mathbb{S} }

\newcommand{\LC}{\left ( }
\newcommand{\RC}{\right ) }
\newcommand{\LD}{\left \{ }
\newcommand{\RD}{\right \} }
\newcommand{\LZ}{\left | }
\newcommand{\RZ}{\right | }
\newcommand{\DS}{\displaystyle }
\newcommand{\LN}{\left \|  }
\newcommand{\RN}{\right \| }
\DeclareMathOperator*{\esssup}{ess\sup}

\DeclareMathOperator{\diag}{diag}
\DeclareMathOperator{\id}{id}
\DeclareMathOperator*{\argmax}{argmax}

\begin{document}

\begin{frontmatter}
\title{Optimal control problem of evolution equation governed by hypergraph Laplacian}

 \author[Fukao]{Takeshi Fukao}
 \ead{fukao@math.ryukoku.ac.jp}

 \author[Ikeda,Ikeda2]{Masahiro Ikeda}
 \ead{masahiro.ikeda@keio.jp/masahiro.ikeda@riken.jp}

 \author[Uchida]{Shun Uchida\corref{Shun}}
 \ead{shunuchida@oita-u.ac.jp}

 \cortext[Shun]{Corresponding author}

 \address[Fukao]{
	Faculty of Advanced Science and Technology, 
	Ryukoku University,\\
	1-5 Yokotani, Seta Oe-cho, Otsu, Shiga,
	520-2194, Japan.}
 \address[Ikeda]{
	Department of Mathematics, 
    Faculty of Science and Technology,   
	Keio University, 
	3-14-1 Hiyoshi Kohoku-ku, Yokohama,
	223-8522, Japan.}
 \address[Ikeda2]{
	Center for Advanced Intelligence Project, RIKEN, 
	Tokyo, 103-0027, Japan.}
 \address[Uchida]{
	Faculty of Science and Technology,  
	Oita University,\\
	700 Dannoharu, Oita, 
	870-1192, Japan.}

\begin{abstract}
In this paper, we consider an optimal control problem of an ordinary differential inclusion 
governed by the hypergraph Laplacian, 
which is defined as a subdifferential of a convex function and then is a set-valued operator.
We can assure the existence of optimal control for a suitable cost function 
by using methods of a priori estimates established in the previous studies.
However, due to the multivaluedness of the hypergraph Laplacian, 
it seems to be difficult to derive the necessary optimality condition for this problem. 
To cope with this difficulty, 
we introduce an approximation operator based on the approximation method of the hypergraph, so-called ``clique expansion.'' 
We first consider the optimality condition of the approximation problem with the clique expansion of the hypergraph Laplacian
and next discuss the convergence to the original problem.
In appendix,  we state some basic properties of the clique expansion of the hypergraph Laplacian for future works.
\end{abstract}


\end{frontmatter}


\section{Introduction}

\subsection{Definition and Background of Hypergraph Laplacian} 

The {\it (weighted) hypergraph} is defined as a triplet $G = (V, E, w)$ of
\begin{itemize}
\item a finite set $V = \{ 1 , 2 , \ldots, N  \}$,

\item a family $E \subset 2^{V} $ of subsets with more than one element of $V$,
that is, $\# e \geq 2  $ for each $e\in E$,

\item a function  $w : E \to (0, \infty)$.

\end{itemize}
This $G$ can be interpreted as a model of a network in which 
the vertices numbered from $1$ to $N$ 
are connected by each hyperedge $ e \in E$.
When $e\in E $ consists of just two elements, then $e = \{ i , j \}$ corresponds to a line segment
connecting $i$-th and $j$-th vertices. 
Hence the {\it usual graph,} which includes the nodes and usual edges, 
can be represented by $G$ with $E $ satisfying $ \# e =2  $ for every $ e \in E$.
The hypergraph is a generalization of the usual graph 
which allows the grouping of multiple members.
For instance, 
the model of relationship of co-authorship between researchers
and communities in the social media
can be described by the hypergraph (see Figure \ref{Graphs}).

\begin{figure}[h]
 \centering
 \includegraphics[keepaspectratio, scale=0.40]
      {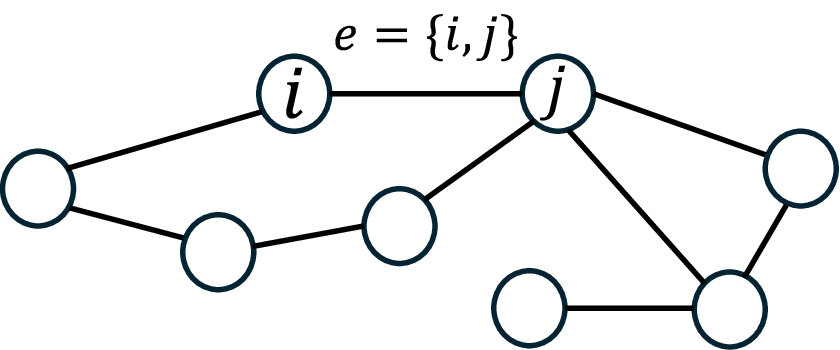}
      \hspace{4mm} 
 \includegraphics[keepaspectratio, scale=0.40]
      {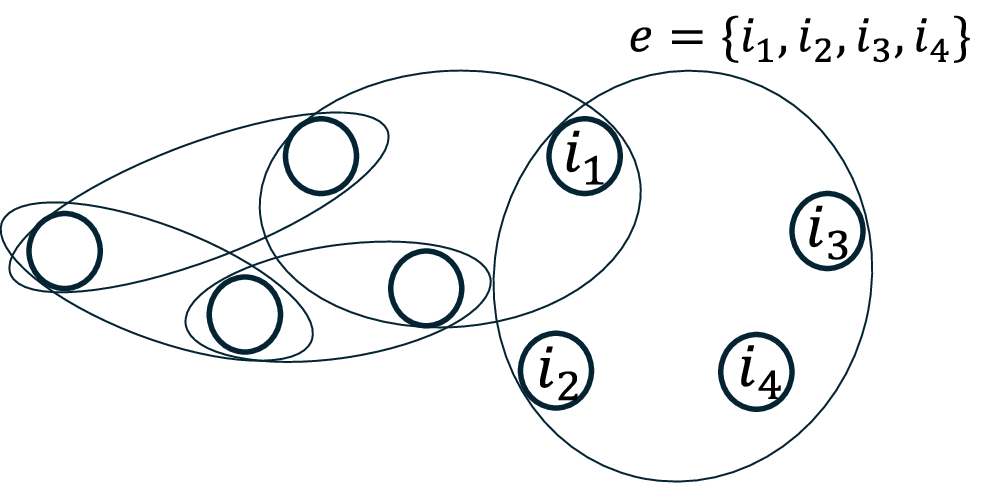}
 \caption[Usual Graph and Hypergraph]{
$V = \{ 1 ,\ldots, N \}$ corresponds to the set of vertices labeled from $1$ to $N$.
Each edge $e \in E $ represents a line segment or a grouping connecting nodes 
included in $e$. 
Throughout this paper, we say $G = (V ,E , w )$ is a {\it usual graph}
if every $e \in E  $ consists of just two elements,
i.e., $G $ represents the left figure
in order to distinguish from the hypergraphs
which genuinely possess hyperedges with  more than two elements (right figure).
}
 \label{Graphs}
\end{figure}

It is well known that we can define the graph Laplacian on the usual graph by the following way.
Let 
\begin{equation}
\label{usu-wei} 
w_{ij} := \begin{cases}
~~
w( \{ i , j \} ) ~~& ~~ \text{ if } \{ i, j \} \in E ,  \\
~~~~~0 ~~& ~~ \text{ if } \{ i, j \} \not \in E , 
\end{cases}
~~~~
d_ i := 
\sum_{j=1}^{N} w _{ ij } 
\end{equation}
and define square matrices of order $N$
by $W := (w _{ij} ) _{ 1 \leq i , j \leq N }$ and $D := \diag ( d_1 ,\ldots , d_ N  )$.
Since 
$w _ {ij} >0 $ if $i$-th and $j$-th vertices are connected
and 
$w _ {ij} = 0 $ if $i$-th and $j$-th vertices are disconnected,
the matrix $W$, which is called (weighted) adjacency matrix,
represents how the vertices are connected to each other in the graph.
On the other hand, 
the matrix $D $, which is called (weighted) degree matrix, describes 
the (weighted) number of edges attached to each vertex.
Then by considering 
that the weight on the edge $w_ {ij} $ implies the preference for the selection of the pass,
we can obtain 
the matrix $D ^{-1} (D -A )$ as the transition matrix of the random walk on this graph.
Here $L := D -A $, which essentially describes the movement of the particles,
is called the graph Laplacian.

The network structure of the usual graph can be investigated 
through the study of eigenvalues and eigenvectors of the graph Laplacian, 
which is called ``spectral graph theory'' established in the 1980s.
This theory has been applied to the algorithm of measuring the importance of website,
which is called PageRank, and the Cheeger type inequality, 
which is related to the cluster analysis (see, e.g., \cite{BP, Ch97, Ch07, CD80} and references therein).
In order to develop the spectral graph theory to more general networks, 
the Laplacian on hypergraphs has been proposed in several recent articles. 
In this paper, we focus on the definition by Yuichi Yoshida
(see \cite{Yoshida00}),
which is based on the theory of the discrete optimization problems and submodular functions
(other definition of the Laplacian on hypergraphs can be found in, e.g., \cite{Jost}).

We here state the definition of hypergraph Laplacian by \cite{Yoshida00} (see also \cite{F-I-U, I-U}).
For each $e \in E$, we define $f_e : \R ^N \to [0 , \infty ) $ by 
\begin{equation}
\label{Def_fe} 
	f_e (x) 
:= \max _{ i, j  \in e} | x _ i  - x _ j  |  
\end{equation}
and 
\begin{equation}
\varphi _{G, p } (x)
:= \frac{1}{p} \sum_{e \in E} w (e) (f _e (x)) ^p ~~~~~p\in [1 ,\infty ) .
\label{Def_phi} 
\end{equation}
When we regard the $i$-th component $x_ i $ of $x \in \R ^N $
as the heat or the number of particles on the $i$-th node of $V$,
then $f_e $ represents the heat/density gradient on the connection $e \in E$
and $\varphi _{G, p }$ can be regarded as a kind of the $p$-Dirichlet energy on the hypergraph $G$.

If $G$ is the usual graph, these can be written  by 
\begin{equation}
\label{usu_fe} 
	f_e (x) 
:=  | x _ i  - x _ j  |  ,~~~~\text{where } e = \{ i, j \},
\end{equation}
and 
\begin{equation}
\label{usu_phi} 
\varphi _{G, p } (x)
:= \frac{1}{2p} \sum ^N _{i , j = 1} w _ {i j} |x_ i  - x_ j | ^p ,
\end{equation}
where $w_ {ij }$ is defined in \eqref{usu-wei}.
Hence we can show that  $\varphi _{G, p }  : \R ^N \to \R $ with $p >1$ is Fr\'{e}chet differentiable 
and especially the derivative $D \varphi _{G, 2 } $ coincides with  the usual graph Laplacian $ L = D -A $.
On the other hand, 
if $G$ essentially possesses a hyperedge with more than two nodes,
$ f_e $ and $ \varphi _{G ,p }$ are not differentiable 
on
$ \bigcup _{ i, j  \in e } \{  x \in \R ^N ;~x_i  = x_ j \}$
and 
$\bigcup _{e \in E } \bigcup _{ i, j  \in e } \{  x \in \R ^N ;~x_i  = x_ j \} $, 
respectively,
due to the singularity of derivative of the max-function.

Since $ f_e $ and $\varphi _{G ,p }$ are continuous and convex on $\R ^N $, however,
we can define the subdifferential operator of them.
Here when $\phi  : \R ^N \to ( - \infty , + \infty ]$
is a proper ($\phi  \not \equiv + \infty $) lower semi-continuous and convex function,
its subgradient at $x \in \R ^N  $ is defined by 
\begin{equation}
\label{sub}
\partial \phi  (x) := \{ \eta \in \R ^N ;~ 
\eta \cdot (z - x ) \leq \phi  (z) - \phi  (x)
 ~~\forall z \in \R ^ N   \} 
\end{equation}
and $ \partial \phi : \R ^N \to 2 ^{\R ^N }$ is called the subdifferential  operator of $\phi $
(basic properties of the subdifferential can be found in, e.g., \cite{Bar, Bre, Rock}).
The subdifferential of $f_e  $ and $\varphi _{G, p }  $
can be represented by 
\begin{align}
\partial f_ e  (x)  &=
\argmax _{b \in B_e} b \cdot x  = \left\{ b_e  \in B_e ;~~b_e \cdot x = \max _{b \in B_e} b \cdot x \right\}  , \label{fe}  \\
\partial \varphi _{G, p} (x) 
&=  \sum_{e \in E } w(e) (f_e (x) ) ^{p-1} \partial   f_e (x) \notag \\
	&= \LD \sum_{e \in E } w(e) (f_e (x) ) ^{p-1} b_e ;~b_e \in \argmax _{b \in B_e } b \cdot x  \RD , \label{Laplacian} 
\end{align}
where $B _e \subset \R ^N $  is defined by 
\begin{align}
B_e & := \text{conv}  \{ \bm{1} _{i }- \bm{1} _ {j} ;~ i , j  \in e  \}   , \notag   \\
&=  \text{conv}  \LD (\ldots , 0 , \stackrel{i}{\stackrel{\vee}{1}},0, \ldots  ,0 , \stackrel{j}{\stackrel{\vee}{- 1}},0,  \ldots  )
\in \R^ N  ; ~ i ,j   \in e \RD \label{BP}  
\end{align}
and $\bm{1} _ {i } $ is the $i$-th unit vector of the canonical basis of $\R ^N $.
Note that 
$\partial f_e (x)$ and 
$\partial \varphi _{G,p } (x) $
genuinely return set-values
on
$ \bigcup _{ i, j  \in e } \{  x \in \R ^N ;~x_i  = x_ j \}$
and 
$\bigcup _{e \in E } \bigcup _{ i, j  \in e } \{  x \in \R ^N ;~x_i  = x_ j \} $, 
respectively.

The subdifferential $ \partial \varphi _{G, p}  : \R ^N \to 2 ^{\R ^N } $
is called {\it hypergraph ($p$-)Laplacian}, where $1 \leq p < \infty $.
This nonlinear multivalued operator is applied to
study the Cheeger like inequality,
the cluster structure, and the PageRank of network represented
by the  hypergraph in information science (see e.g., \cite{FSY, IMTY, L-M, TMIY}).
In \cite{I-U}, we consider the nonlinear set-valued ODE $ x'(t) + \partial \varphi _{G ,p } (x (t)) \ni 0$
and discuss the asymptotic behavior of solutions  via the Poincar\'{e}--Wirtinger type inequality.

\subsection{Setting of Constraint Problem and Optimal Control Problem}

In \cite{F-I-U}, 
we set $N = n + m $
and consider 
the ODE governed by the hypergraph Laplacian $ \partial \varphi _{G ,p } $
under the situation where
the heat at vertices labeled from  $n + 1 $ to $n+ m $ are determined by some given functions
$a _ j : [0, T ] \to \R $ ($j=1 , \ldots, m$), namely, 
with the constraint condition 
\begin{equation}
\label{constraint} 
x_{n+j} (t) = a _ j (t) ~~~~~\forall t \in [0, T ] ~~~~~j=1,\ldots ,m  
\end{equation}
(see Figure \ref{Givendata}).
This problem can be reduce to the evolution equation governed by the subdifferential of the indicator function as follows.
Let the given data be unified by $a : [0 , T ] \to \R  ^N $ as
\begin{equation}
\label{a} 
a (t) := ( 0, \ldots , 0 , a_1 (t) , \ldots , a_ m (t)) ~~~~t\in [0,T ] 
\end{equation}
and the family of them be written by 
\begin{equation}
\mathscr{C}
=
\LD
a  \in  W^{1, 2 } (0, T ; \R ^N  ) ; ~~ 
a (\cdot ) = ( 0, \ldots , 0 , a_1 (\cdot ) , \ldots , a_ m (\cdot  )) 
\RD .
\label{control-set} 
\end{equation}
Here and henceforth, we shall use 
$L^r (0 ,T ; \R ^N )$ 
and 
$ W ^{1, r} (0,T ; \R ^N )$ 
in order to denote the standard Lebesgue and Sobolev space, respectively:
\begin{align*}
&\DS L^{r} (0, T ; \R ^V) := \LD  g  : ( 0, T)  \to \R ^ N  ; 
~
\begin{matrix}
 g \text{ is Lebesgue measurable  }  \\[2mm]
\DS \text{and } \int_{0}^{T}  \| g(t) \| ^r  dt < \infty .
\end{matrix}
 \RD , ~r \in [1 , \infty ) , \\
&\DS L^{\infty } (0, T ; \R ^N )  := \LD  g 
 : ( 0, T)  \to \R ^ N  ; 
~
\begin{matrix}
 g \text{ is Lebesgue measurable   } \\[2mm]
\DS \text{and } \esssup _{ 0 < t  < T }  \| g(t) \|   < \infty .
\end{matrix}
 \RD  , \\[2mm]
&\DS W^{1, r } (0, T ; \R ^N ) := \LD  g \in L^r (0, T ; \R ^N );
			~~  g' \in L^r (0, T ; \R ^N )   \RD ,~~~ r \in [1 , \infty ] ,
\end{align*}
where $g'  := \frac{d}{dt} g  $ is the time derivative of $g$ in the distributional sense.
Define  the constraint set 
 $K_a (t) \subset \R ^N $ with $a \in \mathscr{C}$ and $t \in [0,T ]$
by 
\begin{equation*}
K_a (t) :=
 \LD x \in \R ^N ;~~~
\begin{matrix}
x= (x _1 , \ldots , x _n , a_1 (t) ,\ldots , a_m (t)  ), \\[2mm]
x_i \in \R ,  ~~~i =1,\ldots ,n . 
\end{matrix}
\RD ,
\end{equation*}
which is an affine set where  
the former $n$ components $x_ 1 ,\ldots , x_n $ are chosen freely 
and the latter $m$ components $x_ {n+ 1} ,\ldots , x_{n + m } $ are fixed
by given data $a_ 1 (t ) , \ldots , a_m (t)$.
Set  the indicator on $K _a (t) \subset \R ^N $ by 
\begin{equation}
\label{Indi} 
I_{K_a (t)} (x)=\begin{cases}
~~0 ~~&~~\text{ if } x \in K_a  (t) , \\
~~ + \infty ~~&~~\text{ otherwise.}
\end{cases}
\end{equation}
Then the problem with the constraint  condition \eqref{constraint} 
can be represented by the following Cauchy problem of  evolution equation:
\begin{equation}
~~~~~~~
\begin{cases}
~~\DS x'(t) + \partial \varphi _{G, p } (x (t)) + \partial I_{K _a (t )} (x (t) ) \ni h (t), ~~~t \in (0, T) ,  \\
~~x(0) = x_ 0 ,
\end{cases}
\label{Original-Problem}
\tag{P} 
\end{equation}
where
$ x _ 0 \in \R ^ N $ is the initial state
and $h : [0,T ] \to \R ^ N $ is the given external force.
The well-posedness of \eqref{Original-Problem} has been assured  in \cite{F-I-U} as follows:
\begin{Pro}
\label{Well-posedness-Original} 
Let  $T < \infty$, $a\in \mathscr{C} $,  $x_ 0 \in K_a (0)$,  and $h\in L ^{2} (0, T ; \R ^N )$. 
Then \eqref{Original-Problem} possesses
 a unique solution $x \in W ^{1,2 } (0,T ; \R ^N )$ satisfying $x (t) \in K_a (t)$ for every $ t \in [ 0, T ]$.
\end{Pro}

Based on this fact, we define the solution operator 
$\Lambda : \mathscr{C} \to W ^{1, 2 } (0 ,T ; \R ^N) $ by the relationship
\begin{equation}
\label{SolOp-Original} 
\Lambda (a) =x ~: \text{ the solution to \eqref{Original-Problem} with the given data } a \in \mathscr{C}
\end{equation}
and the cost function 
$J : \mathscr{C} \to [ 0 , \infty )  $ by the relationship
\begin{equation}
\label{Cost-Original} 
J (a ) = 
\frac{1}{2}
\int_{0}^{T} \|  \Lambda (a )  (t) - x _{\ast} (t) \| ^2 dt 
+  
\frac{1}{2}
\int_{0}^{T} \|  a (t) \| ^2 dt ,
\end{equation}
where and henceforth $  x _{\ast} \in L ^2 (0, T ; \R ^N )$ stands for the given target function.
In this paper, 
we consider the optimal control problem 
of the nonlinear set-valued ODE governed by the hypergraph Laplacian \eqref{Original-Problem}
in which we find the minimizer $a^ {\star } \in \mathscr{C} $ of the cost function \eqref{Cost-Original}.

\begin{figure}[H]
 \centering
 \includegraphics[keepaspectratio, scale=0.5]
      {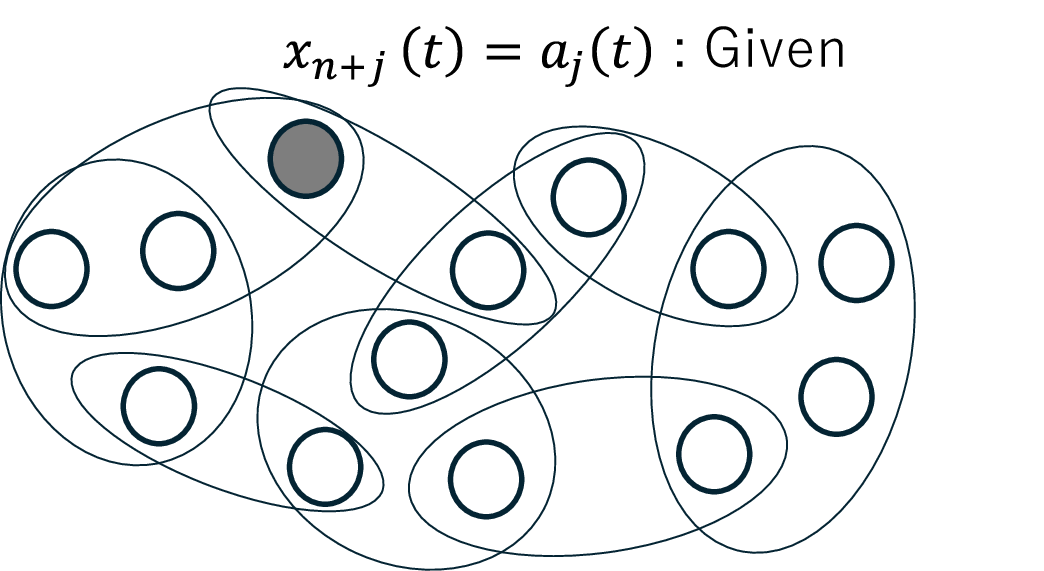}
 \caption[Given data]{We consider the case where 
the heat on the vertices numbered from $ n+ 1$ to $n+ m $ 
are given as $x _{n + j } (t) = a _j (t)$ ($j=1 , \ldots,  m$).
This can be interpreted as a situation where the observer internally manipulates 
the heat of the network. 
One may be able to 
divide the hypergraph $G$ into groups ``close to'' and ``far from''  the controlled nodes
(colored in the above)
 by investigating how to ease to control it.
}
 \label{Givendata}
\end{figure}

\subsection{Existence of Optimal Control}

Since  a priori estimates of the solution to \eqref{Original-Problem} 
has been already established in our previous study \cite{F-I-U},
we can assure the existence of optimal controls of our problem,
i.e., global minimizers of the cost function \eqref{Cost-Original}.
Let the admissible set of controls with parameter $M > 0 $ be denoted by 
\begin{equation}
\label{admissible}
\mathscr{U} ^M _{{\rm ad}} := 
\LD  a \in \mathscr{C } ; ~~ \int_{0}^{T } \| a ' (t) \| ^2 dt \leq M \RD  .
\end{equation}

\begin{Pro}
\label{Exi-OC-Existence} For any $M > 0$, there exists at least one $a ^ {\star } \in \mathscr{U} ^M _{{\rm ad}}  $
such that 
\begin{equation}
\label{minimize-original}
J (a ^ {\star } )  = \min _{a  \in \mathscr{U} ^M _{{\rm ad}} } J (a ) . 
\end{equation}
\end{Pro}

\begin{proof}
Let $\{ a^ k \} _{ k \in \N }$ be a minimizing sequence of $J$. 
By the definition of the cost function  $J $ and the admissible set $ \mathscr{U} ^M _{{\rm ad}}  $,
we can see that  $\{ a^ k \} _{k \in \N }$ is uniformly bounded with respect to the norm of  $W^{1, 2 } ( 0, T ; \R ^N ) $
and there exists a subsequence which strongly converges in $C ([ 0, T ] ; \R ^N )$.
We here omit relabeling and denote the limit by $a ^ {\star }$, namely, let 
\begin{equation*}
a ^ k \to a ^ {\star } ~~\text{ strongly in } C([0, T]; \R ^N )
\text{ and  weakly in } W ^{1 ,2 }( 0, T; \R ^N )
\end{equation*}
as $k \to \infty$.
Set $ x ^k : = \Lambda (a ^ k )$, i.e.,  let $ x ^k $ be a unique solution to
\begin{equation}
\label{Sec1-1} 
\begin{cases}
~~\DS \frac{d}{dt}  x^ k (t) + \partial \varphi _{G, p } (x^k (t)) + \partial I_{ K _{a^ k } (t )} (x^k  (t) ) \ni h (t), ~~~t \in (0, T) ,  \\[2mm]
~~x^k (0) = x ^k _ 0 = (x_{01} ,\ldots , x_{0n } , a ^k _1 (0) , \ldots , a ^k _m (0) )  .
\end{cases}
\end{equation}
Recall  that
we impose  $x^k (0) \in K _{a^ k } ( 0  )  $ on the initial data in Proposition \ref{Well-posedness-Original}.
Let $\xi ^k \in \partial I _{K _{a^k} (\cdot  )} (x^k )$
and $\eta ^k \in  \partial \varphi _{G,p }(x^k ) $
stand for the sections of 
$\partial I _{K _{a^k} (\cdot  )} (x^k )$
and $\partial \varphi _{G,p } (x^k ) $ satisfying \eqref{Sec1-1},  that is, 
suppose that $\xi ^k $ and $\eta ^k $ satisfy 
\begin{equation*}
 \frac{d}{dt}  x ^k (t)  + \eta ^k (t) + \xi ^k (t) =h (t)
\end{equation*}
and 
 $\xi ^k (t) \in \partial I _{K _{a^k} (t  )} (x^k (t) )$
and $\eta ^k (t)  \in  \partial \varphi _{G,p }  (x^k (t)  ) $
for a.e. $t \in ( 0 , T ) $.

For the sake of completeness of this paper,
we here check a priori estimates established in \cite{F-I-U}.
Henceforth, $C_ 1 $ will denote  a general constant which is independent of the index $k$. 
We first multiply \eqref{Sec1-1} by $ x ^k -a ^ k$.
In \cite{F-I-U}, we prove that 
for any $t \in [0,T]$ and $z \in  K_a (t)$, 
the subdifferential of $I _{K_ a (t)}$ can be characterized by 
\begin{equation}
\label{Const-sub} 
\partial I _{K_ a (t)} (z ) =
 \LD \xi  \in \R ^N  ;~~~
\begin{matrix}
\xi= (0,  \ldots , 0 , \xi _{n+1} , \ldots , \xi _{n+m}), \\[2mm]
\xi _ {n+j } \in \R , ~~~j =1 ,\ldots , m. 
\end{matrix}
\RD .
\end{equation}
Hence we have 
\begin{equation*}
\xi ^k (t) \cdot (  x ^k   (t) -a ^ k (t ) ) = 0 ~~\text{ for a.e. } t\in (0, T )
\end{equation*}
since $ x^ k (t) \in   K_a (t) $ and  
\begin{align*}
&x ^ k  (t) -a ^ k (t )  \\
=&
( x ^k _1 (t) , \ldots, x ^ k _n (t) , a ^k _1 (t ),  \ldots , a ^k _ m (t )   )
-
( 0 , \ldots, 0 , a ^k _1 (t ),  \ldots , a ^k _ m (t )   ) \\
=& 
( x ^k _1 (t) , \ldots, x ^k _n (t) , 0, \ldots , 0  ). 
\end{align*}
By the definition of the subgradient \eqref{sub}, we obtain 
\begin{align*}
& \frac{1}{2} \frac{d}{dt}  \|  x ^k (t) - a ^k (t) \| ^2   + \varphi _{G,p } ( x ^k (t)) 
- \varphi _{G,p } ( a ^k (t)) \\
&\hspace{3cm}    \leq \LN  h (t) - \frac{d}{dt } a ^k (t)  \RN \|  x ^k (t) - a ^k (t) \|  .
\end{align*}
We here note that  $ f_ e (z) \leq 2 \| z  \|$ and 
\begin{equation}
\label{bound}
\varphi _{G,p }  (z ) \leq \kappa _{G,p }  \| z \| ^p , 
\hspace{5mm} 
\| \eta \| \leq \kappa _{G,p }   \| z \| ^{p-1 } 
\end{equation}
hold 
 for every $ z \in \R ^V$ and $\eta \in \partial \varphi _{G,p } (z )$, 
where
$ \kappa _{G,p }   $ is a positive constant
which depends only on $ N=n+  m  , E , w  $ and $ p $.  
By this fact together with the boundedness of $\{ a^k \} _{k \in \N}$, 
we get $ \sup_{ 0 \leq t \leq T } \varphi _{G,p } ( a ^k (t)) \leq C_ 1  $.
Then  the Gronwall inequality and $ \int_{0}^{T} \|   \frac{d}{dt } a ^k (t) \|  ^2 dt \leq M  $ lead to  
\begin{equation}
\label{Sec1-2} 
\sup _{0 \leq t \leq T} \|  x ^k  (t)  \|   \leq C_1 .
\end{equation}
By \eqref{bound}, we also have 
\begin{equation}
\label{Sec1-2-1} 
\sup _{0 \leq t \leq T} \|  \eta ^ k  (t) \|   \leq C_1 .
\end{equation}
Next multiplying \eqref{Sec1-1} by
\begin{align*}
\frac{d}{dt} (  x ^ k  (t) -a ^ k (t )  )
=
\LC \frac{d}{dt}  x ^k _1 (t) , \ldots, \frac{d}{dt} x ^k _n (t) , 0 ,  \ldots , 0  \RC  
\end{align*}
and using \eqref{Const-sub} again,
we have 
\begin{equation*}
\xi ^k (t) \cdot \frac{d}{dt} (  x ^k   (t) -a ^ k (t ) ) = 0 ~~\text{ for a.e. } t\in (0, T )
\end{equation*}
and then 
\begin{equation*}
   \LN   \frac{d}{dt} (  x ^k (t) - a ^k (t) ) \RN    
 \leq \LN  h (t) - \frac{d}{dt } a ^k (t) - \eta ^k (t)  \RN  .
\end{equation*}
From $ a ^k \in \mathscr{U} ^M _{\rm ad}$ and \eqref{Sec1-2-1},
we can derive 
\begin{equation}
\label{Sec1-2-2} 
\int_{0}^{T} \LN  \frac{d}{dt}  x ^ k  (t) \RN ^2 dt   \leq C_1 .
\end{equation}
By the equation, we directly get 
\begin{equation}
\label{Sec1-2-3} 
\int_{0}^{T} \|  \xi ^ k  (t) \| ^2 dt   \leq C_1 .
\end{equation}

Now we discuss the convergence as $k \to \infty$. 
Due to \eqref{Sec1-2} and \eqref{Sec1-2-2},
we can apply the Ascoli-Arzela theorem 
and extract  a subsequence of $\{ x ^k \} _{ k \in \N }$
which strongly converges in 
$C ( [ 0, T ] ; \R ^N ) $
(we still employ the same index).
Let $x ^{\star } \in C([0,T]; \R ^N )$ be its limit. 
Remark that $x ^{\star } (t) \in K_{a ^{\star } } (t )$ holds for every $t \in [0, T]$
since $x ^k = (x^k _1 (\cdot ) , \ldots , x ^k _n (\cdot) , a^k _1 (\cdot ) , \ldots , a ^k _m (\cdot)  ) $
and $a ^ k \to a ^{\star } $ strongly in  $C ( [ 0, T ] ; \R ^N ) $.
Then \eqref{Sec1-2-2} yields 
\begin{equation*}
\frac{d}{dt} x ^k  \rightharpoonup \frac{d}{dt} x^{\star }   ~~~\text{weakly in } L^2(0,T; \R^ N ).
\end{equation*}
Moreover,  \eqref{Sec1-2-1}  and \eqref{Sec1-2-3} imply  that
there exist some $ \eta  ^{\star }  , \xi ^{\star }   \in  L^2(0,T; \R^ N )$
such that
\begin{equation*}
\eta ^k  \rightharpoonup  \eta ^{\star }   ,
  ~~~\xi ^k \rightharpoonup  \xi   ^{\star }  
  ~~~~
\text{weakly in } L^2(0,T; \R^ N ) .
\end{equation*}
These obviously
satisfy the equation $ \frac{d}{dt} x ^{\star } (t) + \eta ^{\star }  (t)+ \xi ^{\star } (t) = h(t) $ for a.e. $t \in (0, T )$.
By the demiclosedness of the subdifferential operators, 
 $ \eta ^{\star }    (t) \in \partial \varphi _{G ,p } (x ^{\star } (t) )$ holds for a.e. $t \in (0, T)$. 
We shall show 
  $ \xi ^{\star }    (t) \in \partial  I _{K_ {a ^{\star } } (t) } (x ^{\star } (t) )$ for a.e. $t \in (0, T)$. 
Fix  
$v \in L^2(0,T ; \R ^N )$ satisfying $v(t) \in K _{a ^{\star }} (t)$ for a.e. $t \in (0, T)$ arbitrarily
and define $v ^k  \in L^2(0,T ; \R ^N ) $ by 
\begin{equation*}
v ^k (\cdot )  :=  ( v _1 (\cdot ) , \ldots , v_n (\cdot ),  a^k _1 (\cdot ) , \ldots, a^k _m (\cdot ) ),
\end{equation*}
where  $v  (\cdot )  =  ( v _1 (\cdot ) , \ldots , v_n (\cdot ),  a ^{\star }_1 (\cdot ) , \ldots, a^{\star } _m (\cdot ) ) $.
Clearly, $ v ^k \to v$ strongly in  $C ( [ 0, T ] ; \R ^N ) $
and $ v ^ k (t) \in K _{a ^k} (t)$ for a.e. $t\in (0, T)$.
Then we have by the definition of the subdifferential
\begin{align*}
\int_{0}^{T}  \xi ^k  (t)  \cdot (  v^ k (t) - x ^k  (t)) dt  
&\leq 
\int_{0}^{T}   I_{K_{a^k } (t)} (v ^k (t)) dt -  
\int_{0}^{T}  I_{K_{a^k }  (t)}(x ^k (t)) dt \\
&= 0 .
\end{align*}
By taking its limit as $\lambda  \to 0 $,
we obtain 
\begin{align*}
\int_{0}^{T}  \xi ^\star   (t)  \cdot (  v (t) - x ^\star   (t)) dt  
\leq 0
=  
\int_{0}^{T}   I_{K_{a ^\star  } (t)} (v  (t)) dt -  
\int_{0}^{T}  I_{K_{a ^\star  }  (t)}(x ^ \star  (t)) dt .
\end{align*}
From the arbitrariness of the choice of $v $, we can derive 
$ \xi ^{\star }    (t) \in \partial  I _{K_ {a ^{\star } } (t) } (x ^{\star } (t) )$ for a.e. $t \in (0, T)$
(see our proof of \cite[Theorem 3.1]{F-I-U}).

Therefore, we can assure that $ x ^{\star } = \Lambda (a ^{\star })$, 
in particular, 
\begin{equation*}
\Lambda (a ^k )
\to 
\Lambda (a ^{\star })
  ~~~\text{strongly  in } C( [ 0,T ] ; \R^ N ) .
\end{equation*}
This immediately implies that $ J (a ^k ) \to J (a ^{\star } )$,
whence follows that $ a ^{\star } $ is a global minimizer of the cost function $J$. 
\end{proof}

We can see the existence of the optimal control of $J $ 
by using methods given in our previous paper \cite{F-I-U}.
Therefore in this paper, 
we mainly focus on the necessary optimality condition,
which is an important tool to find the minimizer of the cost function $J$.
In general, when we derive the necessary optimality condition,
 we need the G\^{a}teaux differentiability of the solution operator $\Lambda $ and cost function $J $.
However, we have shown  in  \cite{F-I-U}  that
the solutions to \eqref{Original-Problem} is 
1/2-H\"{o}lder continuous with respect to the given data $a \in \mathscr{C}$  and  this seems to be optimal.
In order to avoid this difficulty, 
we introduce an approximation operator of $\partial \varphi _{G, p }$,
which is based on the clique expansion of the hypergraph.
Since this approximation operator has some interesting properties,
we shall state them not only in section 2 but also in the Appendix.
In Section 3, 
we consider the optimal control problem for the equation
where $\partial \varphi _{G, p }$ is replaced with the approximation operator.
Note that this approximation problem can be regarded as a hybrid control problem 
consisting of the initial data and the external force.
Finally, we discuss the convergence from  the approximation problem 
to the original problem \eqref{Original-Problem}.


\section{Approximation of Hypergraph Laplacian and its Properties}

\subsection{Definition} 
Let $q \geq  1 $ and define 
\begin{equation}
\label{Sec2-1} 
f _{ e,  q } (x) := \LC  \frac{1}{2}  \sum_{i  , j \in e} |x _ i -x _ j | ^q \RC ^{1/q}
=
\LC  \sum_{ i < j \text{ s.t. } i , j \in e} |x_i  - x _j| ^q \RC ^{1/q} 
\end{equation}
and 
\begin{equation}
\varphi _{G, p ,q } (x) := 
\label{Sec2-2} 
\frac{1}{p}
\sum_{e \in E } w(e) (f_{e ,q } (x ) )^p 
=  
\frac{1}{p}
\sum_{e \in E } w(e) \LC  \frac{1}{2}  \sum_{i, j \in e} |x _ i  -x _ j | ^q \RC ^{p/q} . 
\end{equation}
Obviously, the relationship between $f_ {e, q}$ and the original $f_ e$
is similar to that between the $\ell ^q$-norm and $\ell ^{\infty}$-norm on the finite dimensional space.

When $q= p$, we have 
\begin{equation*}
\varphi _{G, p ,p } (x) = 
\frac{1}{2 p}
\sum_{e \in E } w(e)  \sum_{i, j  \in e} |x_i  -x_j | ^p .
\end{equation*}
This can be rewritten  as $\varphi _{G, p ,p } (x) = 
\frac{1}{2 p}
\sum_{ i , j = 1 }^N  w_{ij }   |x_i  -x_j | ^p  $ with 
\begin{equation*}
w_{ij} = w_{ ji  }
:=\begin{cases}
~~\DS  \sum_{e \in E \text{ s.t. } i , j \in e } w(e)
		~~&~~ \text{ if } \exists e \in E \text{ s.t.  } i ,j  \in e ,\\
~~~~~~~~~~~~\DS 0
		~~&~~ \text{ otherwise} .
\end{cases}
\end{equation*}
Hence we can regard $ \varphi _{G ,p , p} $ as the energy function of the usual graph where 
\begin{itemize}

\item $i$-th and $j$-th vertices are connected by the line segment if there exists at least one $e \in E $ 
which includes them  in the original hypergraph and the weight of this edge 
is equal to the sum of weights of $e \in E $ which includes  $i$-th and $j$-th nodes.

\item otherwise, i.e., if there is no $e \in E $ which includes $i$-th and $j$-th vertices in the original hypergraph, 
we do not draw any edge between $i$-th and $j$-th nodes. 

\end{itemize}

Such a usual  graph is called the {\it clique expansion} of the hypergraph, 
where the complete graph is substituted for 
the hyperedge $e\in E$ (see Figure \ref{Clique}).
By considering that the similar situation occurs for $ p \neq q $,
 the operator $\partial  \varphi _{G , p, q } $ can be regarded 
as an approximation of $\partial  \varphi _{G , p} $  based on  the clique expansion of hypergraphs. 
Hence by the arguments of the convergence 
 $\partial  \varphi _{G , p, q } $ tending to the original hypergraph Laplacian  
 $\partial  \varphi _{G , p, } $, 
we can  justify the appropriateness of  the clique expansion as a suitable approximation of hypergraphs
from the analytical point of view.
\begin{figure}[H]
 \centering
 \includegraphics[keepaspectratio, scale=0.5]
      {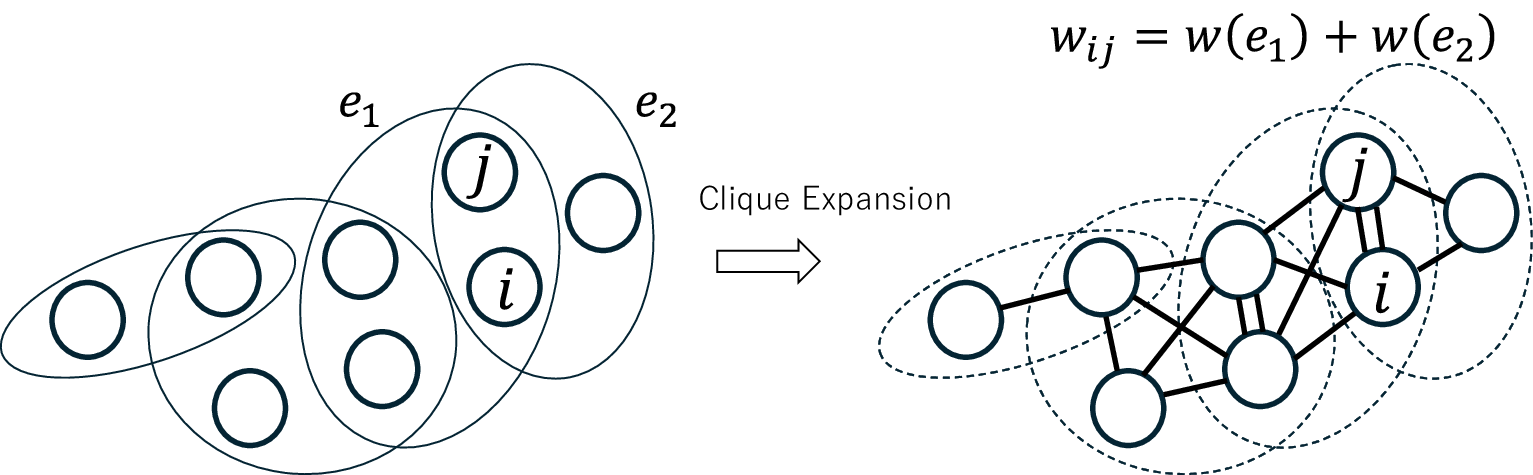}
\caption[Given data]{{\it Clique expansion} is a kind of approximation of  hypergraph 
where the connection by the hyperedge is replaced with the complete graph
(every pair of nodes in $e$ is connected by a line segment).
We can regard $\varphi _{G ,p ,p }$ as a energy of the clique expansion of hypergraph
and 
 the weight $w_ { ij }$ of the line segment connecting $i$-th and $j$-th nodes in the clique expansion
is equal to the sum of the weights of hyperedges including them in the original hypergraph. 
}
 \label{Clique}
\end{figure}

\subsection{Basic Properties} 

It is easy to see that 
by the definition \eqref{Sec2-1} and 
the standard relationship of $\ell ^q $-norms   
\begin{equation}
\label{Sec2-3} 
f_{e, q } (x ) \leq  \sum_{ i < j \text{ s.t. } i , j \in e} | x_i  - x _j |
\leq \# e (\# e -1 ) \| x \| 
\end{equation}
($\# e $ stands for the number of elements of $e$) and then by \eqref{Sec2-2} 
\begin{equation}
\label{Sec2-4}
\varphi _{G , p , q } (x ) \leq \kappa ' _{G ,p } \| x \| ^p ,
\end{equation}
where $\kappa ' _{G ,p } $ is a constant depending only on $ N=n+  m  , E , w  $ and $ p $
and independent of $q$.

Since $ f_{e , q }  $ and  $ \varphi _{ G , p , q } $ are clearly convex and continuous functions, 
we can define the subgradient of them at every $x \in \R ^N $.
Moreover we have for each $l =1, \ldots , N$
\begin{equation}
\label{Formula-Clique} 
\begin{split}
& \partial _{x_l  } \varphi _{G, p ,q  } (x) \\
=& \frac{1}{q}
\sum
 _{ 
e \in E  ~\text{s.t. }
l  \in e
}w(e) 
 \LC
 \frac{1}{2}
 \sum  _{ 
i , j  \in e
}
 | x_ i - x_ j | ^q \RC ^{(p -q ) /q } 
 \partial _{x_ l  }  \LC
 \sum  _{ 
i   \in e }
 | x_ l   - x_ i | ^q \RC  \\
=& 
\sum
 _{ 
e \in E  ~\text{s.t. }
l  \in e
} w(e) 
(f_{e , q } (x) ) ^{ p -q  } 
 \LC \sum  _{ 
i   \in e }
 | x_ l  - x_ i | ^{q-2 }  (x_ l  - x_ i  ) \RC , 
\end{split}
\end{equation}
which is continuous if $ p, q > 1 $.
Hence the subdifferential of  $\varphi _{G ,p, q  }$ with $ p ,q > 1 $
is equal to the total derivative $D \varphi _{G ,p, q  }$ 
since the subgradient of convex function generally coincides with its Fr\'{e}chet derivative if  it is differentiable in usual sense.

As for the boundedness of  $D \varphi _{G ,p, q  }$, there is some constant $\kappa ' _{G , p }$
which is independent of $q $ such that  
\begin{align}
 \| D \varphi _{G, p ,q  } (x)  \| 
\leq & 
\sum
^N  _{ l =1}   
\sum
 _{ 
e \in E  ~\text{s.t. }
l \in e
} w(e) \# e
(f_{e , q } (x) ) ^{ p -1   }  \notag  \\
\leq & 
\kappa ' _{G ,p }  \| x\| ^{p-1  } ~~~~~~~~~~~~~\forall x \in \R ^N 
\label{Sec2-5} 
\end{align}
by
\begin{equation*}
\sum_{i \in e} |x _ k   -x _ i | ^{q-1} 
\leq  (\# e ) ^{1/q } ( f_{e , q } (x) ) ^{q -1 } 
\leq  \# e  ( f_{e , q }  (x )) ^{q -1 } . 
\end{equation*}
By repeating this calculation inductively, we obtain the following fact: 
\begin{Le}
\label{Clique-Smooth} Let  $  p, q >  k $ with some $k  \in \N $.
Then  $\varphi _{G , p, q } : \R ^N \to \R $ is a function of class  $C ^k$
and there exists a constant $ \kappa ' _{ G ,p , q , k }$ which depends on 
$ N=n+  m  , E , w  $ and $ p ,q  , k $ such that 
\begin{equation}
\label{Sec2-6} 
\| D ^ k \varphi _{ G ,p , q } (x) \| \leq  \kappa ' _{ G ,p , q , k } \| x \| ^{p - k }
~~~~~\forall  x \in \R ^N  . 
\end{equation}
Especially, if $ p ,q >1 $, the subdifferential of 
$\varphi _{G , p, q } : \R ^N \to \R $ coincides with $D \varphi _{G , p, q } $ 
and $\kappa ' _{ G ,p , q , 1 } $  is independent of $q $. 
\end{Le}

We next check the convergence of $\varphi _{G , p , q}$ tending to $\varphi _{G , p }$ as $q \to \infty $.
By the relationship of $\ell ^q$-norms of the finite dimensional space (i.e., by the H\"{o}lder inequality), 
we get for any $q _1 < q _2 $
\begin{equation*}
f _{e ,q _2  } (x) \leq f _{e ,q _1  } (x) \leq  \LC \frac{\# e (\# e -1 )}{2 }  \RC ^{ \frac{1}{q_ 1} -   \frac{1}{q_ 2}   } 
f _{e ,q _2  } (x) ~~~~~~~\forall x \in \R ^N . 
\end{equation*}
Hence summing them with respect to $e \in E$, we obtain 
\begin{equation*}
\varphi _{G ,p , q _2 } (x) \leq \varphi _{G ,p , q _ 1 } (x) \leq 
  \nu _E  ^{ \frac{p}{q_ 1} -   \frac{p}{q_ 2}   }   \varphi _{G ,p , q _ 2 } (x) 
  ~~~~~~~
\forall x \in \R ^N ,
\end{equation*}
where $\nu _E := \max _{e \in E} \LC \frac{\# e (\# e -1 )}{2 }  \RC  $.
In particular, when $q_1 = q < \infty $ and $q _2 = \infty $, we have 
\begin{equation}
\varphi _{G ,p } (x) \leq \varphi _{G ,p , q  } (x) \leq 
  \nu _E  ^{ \frac{p}{q }    }  \varphi _{G ,p  } (x) 
  ~~~~~~~
\forall x \in \R ^N .
\label{Sec2-6-1} 
\end{equation}
Therefore the following fact can be assured:
\begin{Le}
\label{Clique-Conver}
Let $ p \in [1 , \infty ) $ and $ 1 \leq q _1 < q_ 2 \leq \infty $.
Then 
\begin{align*}
\LZ  \varphi _{G,p, q_2 } (x)
- 
\varphi _{G,p,q_1  } (x) \RZ 
& \leq ( \nu  _E ^{ \frac{p}{q_ 1} -   \frac{p}{q_ 2}   }  - 1  ) 
\varphi _{G,p , q _2 } (x) \\
& \leq ( \nu  ^{ \frac{p}{q_ 1}  }  _E - \nu  ^{ \frac{p}{q_ 2}  } _E  )  \varphi _{G,p  } (x) 
\end{align*}
holds for any $x \in \R ^N $. Especially, 
if $ q  _1 = q < \infty $ and $ q _2 = \infty $, 
\begin{equation}
\label{Sec2-7} 
\LZ  \varphi _{G,p } (x)
- 
\varphi _{G,p,q  } (x) \RZ 
\leq ( \nu  ^{  \frac{p}{q }   }  _E - 1 )  
\varphi _{G,p  } (x) ~~~~\forall  x \in \R ^N . 
\end{equation}
\end{Le}
From \eqref{Sec2-7} and  \eqref{bound}, 
we can derive the uniform convergence of $\varphi _{G, p , q } $ to $\varphi _{G ,p }  $ as $q \to \infty $ 
on any compact sets of $\R ^N $. 
Particularly, 
the sequence of functionals $\{ \varphi _{G, p , q }  \} _{q  > 1 } $ converges 
to $\varphi _{G ,p }  $ in the sense of Mosco
(see \cite{Mosco1, Mosco2}).

\subsection{Approximated Equations and Their Convergence} 

We shall consider the following approximation problem of \eqref{Original-Problem}
whose main term is replaced by the clique expansion of hypergraph Laplacian:
\begin{equation}
\begin{cases}
~~\DS x'(t) + D\varphi _{G, p , q } (x (t) ) + \frac{1}{\lambda } (Hx(t) - a(t)  ) = h(t) , ~~~t \in (0, T) , \\
~~\DS x(0) = x_ 0 , 
\end{cases}
\label{Approx-Prob}
\tag{${\rm P } ^  { q , \lambda  } $} 
\end{equation}
where $H : \R ^N \to \R ^N $ is the linear operator defined by  
\begin{equation}
\label{H} 
H z
:=
 (0 ,\ldots , 0 , z_{n+1} ,\ldots ,z_{n+m } ),
 ~~~\text{where}~~z=  (z _1  ,\ldots  , z_{n+m }  ) . 
\end{equation}
Note that the third term of L.H.S. of \eqref{Approx-Prob}
is coincides with the Yosida approximation of $\partial I _{K _{a} (t )}$.
In fact, the projection onto the closed convex set  $ K_a (t) \subset \R ^N $
can be written as 
\begin{equation*}
\text{Proj} _{K_a (t)} z = (z_1 ,\ldots , z_n , a_1 (t) ,\ldots , a_m ( t)), 
\end{equation*}
where $z=  (z _1  ,\ldots , z_n , z_{n+1} ,\ldots , z_{n+m }  )$,
and then 
the Moreau-Yosida regularization of $I _{K _{a} (t )}$ and  
the Yosida approximation of $ \partial  I _{K _{a} (t )}$  are  
\begin{align*}
\LC I _{K_a (t)} \RC _{\lambda } (z) &:= \inf _{ y\in \R ^V } \LD  \frac{1}{2\lambda } \| z -y \| ^2 + I_{K_a (t)} (y)  \RD
			= \frac{1}{2\lambda }  \| z - \text{Proj} _{K_a (t)} z \| ^2  , \\[3mm]
\LC \partial I _{K_a (t)} \RC _{\lambda } (z) &= \partial    \LC I _{K_a (t)} \RC _{\lambda } (z) = 
			\frac{1}{\lambda } \LC z - \text{Proj} _{K_a (t)} z \RC \\
	&= \frac{1}{\lambda } \LC (z_1 ,\ldots , z_ n , z_{n+1} , \ldots, z_{n+m}) -  (z_1 ,\ldots , z_ n , a_{1} (t) , \ldots, a_{m} (t)) \RC \\
	&= \frac{1}{\lambda } \LC H z  -  a(t)  \RC 	.
\end{align*}
Since \eqref{Approx-Prob} is governed by the continuous main term 
$D\varphi _{G, p , q } (x (t) ) + \frac{1}{\lambda } Hx(t) $
and the external force $  h( t ) + \frac{1}{\lambda } a(t) $,
we can easily assure that there exists a unique solution to   \eqref{Approx-Prob}  
belonging to $W ^{1, 2 } (0, T ; \R ^N )$
for every $ x _ 0 \in \R ^N $, $a \in \mathscr{C}$, and $h \in L^2  ( 0 ,T ;\R ^N )$.
In this subsection, we show that the solution to this approximation problem \eqref{Approx-Prob}  
tends to that of the original problem \eqref{Original-Problem} as $\lambda  \to 0 $ and $q \to \infty $ as follows:
\begin{Th}
\label{Convergence-to-original-sol} 
Let $ p >1  $ and the  sequences $ \{ q ^ i \} _{i \in \N } \subset [1 , \infty ) $,  
 $ \{ \lambda  ^  i \} _{i \in \N } \subset ( 0 , \infty ) $ satisfy  
$q ^ i \to \infty $, $\lambda ^ i \to 0$ as $i \to \infty $, respectively.
Suppose that 
$ \{ a ^ i \} _{ i \in \N } \subset \mathscr{C} $ and 
 $ \{ h ^ i  \} _{i \in \N } \subset L ^2 ( 0 , T ; \R ^N  ) $ fulfill
\begin{equation}
\label{Assum-Th-S2-1} 
\begin{split}
a ^i \to a ~~& \text{ weakly in  }W ^{1,2} (0, T ; \R ^N ) ,\\
~~& \text{ strongly in   } C ( [ 0, T ] ; \R ^N ) ,\\
h ^i  \to h ~~& \text{ weakly in  } L ^ 2  (0, T ; \R ^N )  ,
\end{split}
\end{equation}
with some 
$ a \in  \mathscr{C} $ and 
 $ h \in  L ^2 ( 0 , T ; \R ^N  ) $.
Moreover,  assume that 
the  sequence of initial data
 $ \{ x ^ i _ 0 \} _{j \in \N } \subset \R ^N $ converging to $x_ 0 $
satisfies 
 $  x ^ i _ 0 \in K _{a ^i  } (0) $ for every $i \in \N $ and  $x_ 0 \in  K _{a  } (0) $.
Then the solution $x ^{ i} $ of
 \eqref{Approx-Prob} 
with $ ( q , \lambda , a , h , x_ 0 ) =   ( q ^i , \lambda ^i  , a ^i  , h ^i , x ^i _ 0 ) $
converges to 
the solution $x$ of \eqref{Original-Problem} with 
$ ( a , h , x_ 0 ) $ given above in the following sense:
\begin{equation}
\label{Res-Th-S2-1} 
\begin{split}
x ^{i} \to x  ~~& \text{ weakly in  }W ^{1,2} (0, T ; \R ^N ) ,\\
~~& \text{ strongly in   } C ( [ 0, T ] ; \R ^N ) .
\end{split}
\end{equation}
\end{Th}

\begin{proof}
Subtracting $ ( a ^i ) '$ from both sides of \eqref{Approx-Prob}, we get 
 \begin{equation}
\begin{split}
 &( x ^i (t) - a ^i  (t) ) ' +D \varphi _{ G, p , q ^i   } (x ^ i   (t) ) \\
 &~~~~~~~~
 +  \frac{1}{\lambda ^ i  } \LC H x ^i  (t) -  a ^i   (t)  \RC  = h ^i  (t)  -  (  a ^i  (t) ) '   .
\end{split}
\label{Sec2-8} 
\end{equation}
Multiplying this equation by 
$x ^i - a^ i $, we obtain  
\begin{align*}
&\frac{1}{2} \frac{d}{dt} \| x ^i (t) - a^ i  (t)  \| ^2 
 + \varphi _{ G, p , q ^i } (x ^i  (t) )  - \varphi _{ G, p , q ^i  } (a ^i  (t) )  \\
&~~~~+ \frac{1}{ \lambda ^i }  \|  H x ^i  (t) -  a ^i   (t)  
 \| ^2
\leq	\LN  \frac{d}{dt} a ^ i  (t) - h ^i (t) \RN \| x ^i  (t) - a^ i  (t)  \|  
\end{align*}
We here use the fact that the inner product of
\begin{equation*}
H x ^i  (t) -  a ^i   (t) 
= 
 (0 ,\ldots , 0 , x ^ i  _{n+1} (t) - a^i  _1 (t) , \ldots , x^i  _{n+m} (t) - a^ i  _m (t) )
\end{equation*}
and 
\begin{equation*}
x ^i  (t) - a^ i (t)
= 
 (x ^ i  _1 (t) ,\ldots , x ^ i  _n  (t)  , x ^ i  _{n+1} (t) - a^i  _1 (t) , \ldots , x^i  _{n+m} (t) - a^ i  _m (t) ) 
\end{equation*}
is equal to 
\begin{align*}
( H x ^i  (t) -  a ^i   (t)  ) \cdot ( x ^i  (t) - a^ i (t) ) 
&= \|  H x ^i  (t) -  a ^i   (t)   \| ^2 \\
&  =
  \sum_{j=1}^{m} |  x^i  _{n+j} (t) - a^ i _j  (t) | ^2  .
\end{align*}
Since $\{ a ^ i \} _{i \in \N }$ is uniformly bounded with respect to the $W ^{1 ,2 } (0, T ;  \R ^N ) $-norm
and the constant in  \eqref{Sec2-4} is independent of $q$, 
we have  
\begin{equation*}
\sup _{ 0 \leq t \leq T} 
\varphi _{ G, p , q ^i  } (a ^i  (t) )  \leq C_ 2 ,
\end{equation*}
where and henceforth, $C_2 $ denotes a general constant 
independent of the index $i$. 
Then the Gronwall inequality yields 
\begin{equation}
\label{Apriori-01}
 \sup _i  \LC  \sup _{ 0 \leq t \leq T} \| x ^ i (t)\| \RC \leq C_ 2 ,
\end{equation}
which also implies that 
\begin{equation}
\label{Apriori-02}
 \sup _i  \LC  \sup _{ 0 \leq t \leq T} \| D \varphi _{G , p , q ^ i }  ( x ^ i (t) ) \| 
+
\sup _{ 0 \leq t \leq T}  \varphi _{G , p , q ^ i } (  x ^ i (t) ) 
\RC \leq C_ 2 ,
\end{equation}
(recall \eqref{Sec2-4} and \eqref{Sec2-6} with $k =1 $).
Next testing \eqref{Sec2-8} by  $( x ^ i - a^ i ) '   $, 
we get 
\begin{align*}
& \frac{1}{2} \LN  ( x ^i - a^ i ) '   \RN  ^2 
		+\frac{d}{dt}  \varphi _{ G, p , q ^i  } (x ^i  (t) )  
	+ 
\frac{1}{ 2 \lambda ^i }  \frac{d}{dt} \|  H x ^i  (t) -  a ^i   (t)   \| ^2  \\
\leq &		
\frac{1}{2} \LN  \frac{d}{dt} a ^i (t) - h ^i (t) \RN ^2 
+\LN  D \varphi _{ G, p , q ^i  } (x ^i (t) ) \RN \LN  \frac{d}{dt} a ^i (t)  \RN   ,
\end{align*}
which leads to 
\begin{align*}
&  \int_{0}^{T} \LN \frac{d}{dt}  x ^i (t)   \RN  ^2 dt 
  +
  \frac{1}{  \lambda ^i }   \sup _{ 0 \leq t \leq T }  \|  H x ^i  (t) -  a ^i   (t)   \| ^2   \\ 
& ~~\leq 
C_2 \LC 1  +  \frac{1}{  \lambda ^ i }  \|  H x ^i  (0) -  a ^i   (0)   \| ^2  \RC
= 
C_2 \LC 1  +  \frac{1}{  \lambda ^ i }   \sum_{j =1 }^{m} |  x^i  _{n+j} (0) - a^{i} _j (0) | ^2  \RC .
\end{align*}
Since we impose 
$x^i (0) \in K_{a ^i} (0) $ on the initial data, 
the second term of the R.H.S. of the above is zero. 
Hence we can derive 
\begin{equation}
\label{Apriori-03}
\int_{0}^{T} \LN \frac{d}{dt}  x^i (t) \RN ^2 dt  + 
  \frac{1}{  \lambda ^i }    \sup _{ 0 \leq t \leq T } \|  H x ^i  (t) -  a ^i   (t)   \| ^2  
\leq C_2 
 \end{equation}
and by the equation, 
\begin{equation}
\label{Apriori-04}
\int_{0}^{T} 
\LN 
 ( \partial I _{K_{a ^ i  } (t) } ) _{\lambda ^i } (x ^i  (t ) )  \RN ^2 dt  
=
\int_{0}^{T} 
\LN 
\frac{1}{\lambda ^i  } \LC H x ^i (t) -  a ^i  (t)  \RC \RN ^2 dt  \leq C_2 .
\end{equation}

By using these uniform boundedness, we discuss the convergence of solution $\{ x ^i \} _{ i \in \N }$. 
From \eqref{Apriori-01} and \eqref{Apriori-03}, there exists a subsequence 
 (we omit relabeling since the whole sequence also converges as will be seen later)
 and some $x \in W ^{1, 2 } (0,  T ; \R ^N )$ satisfying \eqref{Res-Th-S2-1}.
By \eqref{Apriori-03},  we  have 
\begin{equation*}
\lim_{i \to \infty } \LC \sup _{ 0 \leq t \leq T } \|  H x ^i  (t) -  a ^i   (t)   \| \RC  \to  0 ,
\end{equation*}
that is, 
\begin{equation*}
H x(t) = a(t)~~~\Leftrightarrow ~~~
x_{ n+ j  } (t) = a_  j  (t) ~~~\forall j  =1 ,\ldots ,m ,~~\forall t\in [0 ,T ] , 
\end{equation*}
which is equivalent to $ x (t) \in K _a (t) $ for every $ t \in [0, T ]$. 
According to \eqref{Apriori-02} and \eqref{Apriori-04}, there exist 
$ \eta , \xi \in L^2 (0, T ;\R ^N )$ such that
\begin{equation*}
 D \varphi _{G ,p , q ^ i  } (x^ i  ) \to  \eta ,
 ~~~~
  ( \partial I _{K_{a ^ i  } (t) } ) _{\lambda ^i } (x ^  i  )  \to  \xi
  ~~~~~~~
  \text{ weakly in } L^2(0,T ;\R ^N ) . 
\end{equation*}
Taking the limit of the equation \eqref{Approx-Prob} as $ i \to \infty $, 
we obtain 
\begin{equation*}
\begin{cases}
~~\DS x'(t) + \eta (t)  + \xi (t)  = h(t),  \\
~~\DS x(0) = x_ 0 \in K _ a (0) , 
\end{cases}
\end{equation*}
We here check that $ \eta $ and $ \xi  $ are the sections satisfying the original equation \eqref{Original-Problem},
i.e., 
$\eta (t) \in \partial \varphi _{G,p } (x (t))$
and 
$\xi  (t) \in \partial I  _{ K_a (t)  } (x (t))$ hold for a.e. $t\in (0, T )$.
Fix $y \in L^2 (0,T ; \R ^N )  $ such that  $ \int_{0}^{T} \varphi _{G,p } (y (t) ) dt <\infty  $ arbitrarily
(note that 
$\int_{0}^{T} \varphi _{G,p , q ^ i   } (y (t) ) dt <\infty  $ by \eqref{Sec2-6-1}). 
Since $ D \varphi _{G ,p , q ^ i  } $ coincides with the subdifferential  of the convex function $\varphi _{G ,p , q ^ i  } $,
we infer that 
\begin{align*}
&\int_{0}^{T} 
\LC  h ^i (t)  - \frac{d}{dt} x ^i  (t) -  \frac{1}{\lambda ^i   } \LC H x ^i  (t) -  a ^i  (t)  \RC \RC \cdot (y  (t)- x^i  (t))
dt \\
\leq &\int_{0}^{T} \varphi _{ G, p , q ^i   } (y   (t) ) dt  
-
\int_{0}^{T} \varphi _{ G, p , q ^i  } (x ^i   (t) ) dt  .
\end{align*}
From  \eqref{Sec2-6-1} and Lemma \ref{Clique-Conver}, we can apply the Lebesgue dominant convergence theorem 
and derive 
\begin{equation*}
\int_{0}^{T} 
\eta (t )  \cdot (y  (t)- x   (t))
dt 
\leq \int_{0}^{T} \varphi _{ G, p   } (y   (t) ) dt  
- \int_{0}^{T} \varphi _{ G, p  } (x   (t) ) dt  .
\end{equation*}
Then by the abstract theory for the $L^2 (0, T ; \R ^N  )$-realization of the subdifferential operator
(see, e.g, 
\cite[Proposition 1.1]{Ken75}
and 
\cite[Proposition 3]{PUni}), we can assure that 
$\eta (t) \in \partial \varphi _{G,p } (x (t))$ for a.e. $t\in (0, T )$.

Next fix $z \in L ^2 (0, T ; \R ^N ) $ such that 
$z =  (z_1 (\cdot ) ,\ldots z_m (\cdot )  , a_ 1 (\cdot ) , \ldots , a_ m (\cdot ) )  $ arbitrarily 
and let  $z ^i  : =  (z_1 (\cdot ) ,\ldots z_m (\cdot )  , a ^i  _ 1 (\cdot ) , \ldots , a  ^i  _ m (\cdot ) )  $. 
Remark that  
$ z ^ i \to z $ strongly in $C([0,T] ; \R ^N )$
by the assumption \eqref{Assum-Th-S2-1}. 
Multiplying \eqref{Sec2-8}
by $z ^i  - x^i $ and integrating over $[0, T ]$, we have 
\begin{align*}
&\int_{0}^{T} 
\LC  h ^i  (t)  - \frac{d}{dt} x ^ i  (t) - D \varphi _{ G, p , q ^i   } (x ^i   (t) )  \RC  \cdot (z ^ i  (t)- x^ i   (t))
dt \\
\leq &
\int_{0}^{T} (  I _{K_{a ^i   } (t) } ) _{\lambda ^i  } (z ^i  (t)) dt  
-
\int_{0}^{T} (  I _{K_{a ^ i  } (t) } ) _{\lambda ^i  } (x ^i  (t) ) dt \leq 0 ,
\end{align*}
where we use the fact that 
$\frac{1}{\lambda } \LC H x (t)  -  a^ i (t)  \RC \in  \partial   \LC I _{K_{a ^i } (t)} \RC _{\lambda ^i } ( x^ i (t) )  $, 
$ z ^i  (t) \in K_{a ^i   } (t)  $, and 
$ (  I _{K_{a ^i   } (t) } ) _{\lambda ^i  } \geq 0 $. 
Passing the limit as $i \to \infty $, we obtain 
\begin{equation*}
\int_{0}^{T} 
\LC  h (t)  -x'(t) - \eta (t)  \RC  \cdot (z (t)- x (t))
dt \leq 0 ,
\end{equation*}
which is equivalent to
\begin{align*}
& \int_{0}^{T} 
\LC  h (t)  -x'(t) - \eta (t) \RC  \cdot (z (t)- x (t))
dt  \\
& ~~~~~~~~\leq \int_{0}^{T}   I _{K_{a  } (t) }  (z (t)) dt  
-
\int_{0}^{T}   I _{ K_{a } (t) }  (x (t) ) dt   
\end{align*}
by $ x(t) , z  (t) \in K_{a   } (t)  $ for any $t \in [0, T ]$.  
Consequently, we can show that  $\xi  (t) \in \partial I  _{ K_a (t)  } (x (t))$  for a.e. $t\in (0, T )$
by  \cite[Proposition 1.1]{Ken75}
and then $x $ satisfies the all requirements  of the solution to the original problem \eqref{Original-Problem}. 
Since the solution to \eqref{Original-Problem} is unique, 
 the limit $x$ is determined independently of the choice of subsequences 
and thus the whole sequence $ \{ x ^ i \} _{i \in \N }$ also converges to $x$. 
\end{proof}


\section{Optimal Control Problem for Approximated Equation}

We define 
the solution operator 
$\Lambda ^{ q ,\lambda  } : \mathscr{C} \to W ^{1 ,2 } (0, T ; \R ^N )$
by the correspondence of $\Lambda ^{ q ,\lambda  } (a  ) $ with the solution to \eqref{Approx-Prob},
where we impose the initial data on $x_ 0 \in K _a (0)$
so that the convergence as $q \to \infty , \lambda \to 0 $ is valid.
Since the initial data depends on the control $a $ in this setting,
the optimal control problem for this equation 
can be interpreted as a problem which is associated with a hybrid control by the initial data and the external force.
By this analogy, we set the cost function by 
\begin{equation}
\label{Cost-Appro}
\begin{split}
 J ^{q , \lambda } (a ) 
&: = 
\frac{1}{2}
\int_{0}^{T} \|  \Lambda ^{q , \lambda }  (a )  (t) - x _{\ast} (t) \| ^2 dt 
+  
\frac{1}{2}
\int_{0}^{T} \|  a (t) \| ^2 dt  \\
&\hspace{4cm} + 
 \frac{\lambda }{2} \|  \Lambda ^{q , \lambda }  (a) (T) - z _ {\ast}  \|  ^2  
+
 \frac{\lambda }{2}  \|  a (0 ) \|  ^2 ,
\end{split}
\end{equation}
where $z _{ \ast  } $ is a (dummy) target of  final state 
(we later see that $z _{\ast }$ can be chosen independently of $x_ {\ast} (\cdot )$).

We first check that $ J ^{q , \lambda } $ possesses a global minimizer:
\begin{Th}
\label{Exi-OC-Existence-appro} 
Let $ p ,q >1 $ and $\lambda > 0 $.
For any $M > 0$, there exists at least one $a ^ {\star ; q , \lambda ,  } \in \mathscr{U} ^M _{{\rm ad}}  $
(see \eqref{admissible})
such that 
\begin{equation}
\label{minimize-appro}
J  ^{  q , \lambda }  (a ^ {\star ;  q , \lambda } )  = \min _{a  \in \mathscr{U} ^M _{{\rm ad}} } J ^{  q , \lambda } (a ) . 
\end{equation}
\end{Th}
\begin{proof}
Let $\{ a^ k \} _{ k \in \N }$ be a minimizing sequence of $J ^{q , \lambda } $, 
which clearly satisfies the  uniform boundedness in $W^{1, 2 } ( 0, T ; \R ^N ) $. 
Let the limit of (subsequence of) $ \{ a ^ k \} $ be denoted by $ a ^{\star ; q , \lambda }$. 
By repeating exactly the same a priori estimates as that in our proof of Theorem \ref{Convergence-to-original-sol},
we can assure that 
\begin{equation*}
  \LC  \sup _{ 0 \leq t \leq T} \| x ^ k (t)\| \RC
+
\int_{0}^{T} \LN \frac{d}{dt}  x^k (t) \RN ^2 dt  \leq C_ 3 ,
\end{equation*}
where $C_ 3$ is a constant independent of $k$.
By the continuity of $D \varphi _{ G , p,  q } $ and $H $, 
the limit of (subsequence of) $\{ x ^{k} \} _{k \in \N }$ becomes a unique solution to \eqref{Approx-Prob}
with $a = a ^{\star ; q , \lambda } $, whence follows that 
$ a ^{\star ; q , \lambda }  $ is a global minimizer of the cost function $ J ^{  q , \lambda } $.
\end{proof}

We next consider the necessary optimality condition for the approximation problem.
To this end, we show the G\^{a}teaux differentiability of the
approximated solution operator $\Lambda ^{q , \lambda }$ and the cost function $ J ^{q, \lambda } $.

\begin{Th}
\label{Gateaux-solution}
Let $ p ,q >3$. Then 
$\Lambda ^{q , \lambda } : \mathscr{C} \to W^{1, 2 } (0, T ; \R ^N )$ is G\^{a}teaux differentiable
at any point and in any direction.
Moreover,  
the G\^{a}teaux derivative 
${\rm d} \Lambda ^{q , \lambda } (a ; b )$ at $a \in \mathscr{C} $  in the direction  $b  \in \mathscr{C} $ 
coincides with $\Xi ^{ a , b } \in C ^1 ([0, T] ; \R ^ N)$, which is a unique solution to
\begin{equation}
\begin{cases}
~~\DS  \frac{d}{dt} \Xi ^{ a , b } (t) + D^2 \varphi _{G, p , q } (  x (t)  ) \Xi ^{ a , b } (t )
		+ \frac{1}{\lambda } (H \Xi ^{ a , b }  (t)  - b  (t) ) = 0 , \\[3mm]
~~\DS \Xi ^{ a , b } ( 0) 
	=b (0) = (0 , \ldots , 0 , b _1 (0), \ldots ,  b  _m (0)), \\[3mm]
 ~~ x  = \Lambda ^ {q ,\lambda } (a )   .
\end{cases}
\tag{$ {\rm L} ^  { a , b   }$}
\label{LinearizedProb}
\end{equation}
\end{Th}
\begin{proof}
Fix $a , b \in \mathscr{C}$ and define 
\begin{equation*}
 x ^{ s } : = \Lambda ^ {q , \lambda } (a + s b )   ~~~~s \in [ -1 , 1 ]
\end{equation*}
and  $x  : = x ^ 0  =  \Lambda ^ {q , \lambda } (a  )  $. 
By  reprising the same calculation for  a priori estimates in our proof of Theorem \ref{Convergence-to-original-sol},
we have 
\begin{equation*}
 \sup _{-1 \leq s \leq  1 } \LC \sup _{0 \leq t \leq T} \LN  x^ s  (t) \RN \RC 
 +
 \sup _{-1 \leq s \leq  1 } \LC
\int_{0}^{T}  \LN \frac{d}{dt}  x^ s  (t) \RN ^2 dt \RC 
\leq C_4 .
\end{equation*}
Here and henceforth, $C_ 4$ stands for a general constant independent of $s \in [-1 , 1 ]$.

Let $ \overline{x} := x^ s - x  $, which satisfy 
\begin{equation}
\label{Apriori-Sec3-01} 
\begin{cases}
~~\DS \overline{x}  ' (t) + D\varphi _{G, p , q } (x ^s  (t) ) -  D\varphi _{G, p , q } (x (t) )
		+ \frac{1}{\lambda } (H \overline{ x } (t)  - s b  (t) ) = 0  , \\
~~\DS \overline{ x }  (0) = ( 0 , \ldots , 0 ,  s b  _1 (0), \ldots ,  s b   _m (0))
= s b   (0 ) . 
\end{cases}
\end{equation}
Remark that $D\varphi _{G, p , q } $ coincides with the subdifferential of convex function $\varphi _{G, p , q }$.
Then the monotonicity of subdifferential operators yields
\begin{equation*}
\LC 
 D\varphi _{G, p , q } (x ^s  (t) ) -  D\varphi _{G, p , q } (x (t) )
\RC \cdot  \overline{ x } (t)  \geq 0 ~~~~\forall t \in [0, T ] .
\end{equation*}
Furthermore, by the definition, 
\begin{equation*}
H  \overline{ x } (t)  \cdot  \overline{ x } (t)  \geq 0 ~~~~\forall t \in [0, T ] .
\end{equation*}
Hence testing \eqref{Apriori-Sec3-01} by $\overline{x} $, we get 
\begin{equation*}
\frac{1}{2} \frac{d}{dt } \| \overline{x} (t) \| ^2 \leq   \frac{ | s|  }{\lambda } \| b (t) \| \|  \overline{x} (t)  \|  ,
\end{equation*}
which leads to 
\begin{equation}
\label{Apriori-Sec3-02} 
\sup _{0 \leq t \leq T } \| \overline{x} (t) \|  \leq
| s| \| b (0) \|
+
   \frac{| s|}{\lambda } \int_{0}^{T} \| b (t) \| dt 
   \leq C_4 |s |  .
\end{equation}

Next define 
\begin{equation*}
\hat{x} (t) := x ^s  (t)- x (t) - s \Xi ^ {a , b  } (t)  ,
\end{equation*}
which satisfies 
\begin{equation}
\label{Apriori-Sec3-03}
\begin{cases}
~~\DS \frac{d}{dt} \hat{x} (t)
 + D \varphi _{G, p , q } (x ^s  (t) ) -  D\varphi _{G, p , q } (x (t) ) 
  \\
 \DS ~~~~\hspace{3cm} 
		- s D ^2 \varphi _{G, p , q } ( x  (t) ) \Xi ^ {a , b  } (t) + \frac{1}{\lambda } H \hat{x} (t)  = 0 , \\
~~\DS \hat{ x }  (0) = 0 .
\end{cases}
\end{equation}
Here $\varphi _{G, p , q } $  is a function of class  $C ^3$
since we assume $ p, q  >3 $. 
Then the $3$rd-order Taylor expansion is  
\begin{align*}
 &D\varphi _{G, p , q } (x^s  (t) )
  -  D\varphi _{G, p , q } (x (t) )
	- s  D^2 \varphi _{G, p , q } (x (t) ) \Xi ^{a , b }  (t )  \\[2mm]
 =& 
 D^2 
\varphi _{G, p , q } (x (t) ) (x ^s  (t) - x (t)) 
 + D ^3 
 \varphi _{G, p , q } ( \Theta  (t) ) (x^s  (t) - x (t)) ^2  \\[2mm]
& \hspace{3cm }  	- s  D^2 \varphi _{G, p , q } (x (t) ) \Xi ^{a , b }  (t )   \\[2mm]
  =& 
 D^2 
\varphi _{G, p , q } (x (t) ) \hat{x}(t) 
 + D ^3 
 \varphi _{G, p , q } ( \Theta  (t) ) (x ^s  (t) - x (t)) ^2  ,
\end{align*}
where $ \Theta :[ 0, T] \to \R ^N $ satisfies $ \Theta (t) = \tau x^s  (t) + (1 - \tau ) x  (t)$
with some $\tau = \tau (t) \in [ 0,1 ]$ for each $t \in [0,T ]$. Hence 
\begin{equation*}
\sup _{0 \leq t \leq T }  | \Theta (t) | \leq 
\sup _{0 \leq t \leq T } \max\{ | x ^s  (t)| , |x(t) | \} \leq C_4 .
\end{equation*}
We also recall that $ \varphi _{G, p, q }$ is a convex function and then 
 $D ^2  \varphi _{G, p, q } (x (t) )$  is a non-negative matrix.
 Therefore, 
by testing \eqref{Apriori-Sec3-03} by $\hat{x}$, we can derive from \eqref{Apriori-Sec3-02}
\begin{align*}
\frac{1}{2} \frac{d}{dt} \| \hat{x } (t) \| ^2
&\leq  
\LC \sup _{0 \leq t \leq T}   \| D ^3 
 \varphi _{G, p , q } ( \Theta  (t) ) \|  \RC \|   \overline{  x }  (t)   \|  ^2 \| \hat{x } (t)  \| \\
& \leq 
C_4 |s | ^2  \| \hat{x } (t)  \| ,
\end{align*}
that is, 
\begin{equation*}
\sup _{0 \leq t \leq T } \| \hat{x} (t) \|  
\leq  C_4 
|s| ^2 .
\end{equation*}
We also obtain 
\begin{align*}
& \sup_{0\leq t \leq T }\LN \frac{d}{dt} \hat{x} (t) \RN  
 \leq 
\sup_{0\leq t \leq T }\| D ^2 \varphi _{G, p , q } ( x (t) ) \| \| \hat{ x } (t) \|  \\
&\hspace{2cm} +
\sup_{0\leq t \leq T }\| D ^3 \varphi _{G, p , q } (\Theta (t) ) \|  \| x ^s   (t) - x (t) \| ^2 
+ \frac{1}{\lambda } \| \hat{x} (t) \| 
\leq 
 C_{4} |s | ^2 .
\end{align*}
Therefore 
\begin{equation}
\label{Apriori-Sec3-04} 
\begin{split}
\LN 
\frac{\Lambda ^{q , \lambda } (a + s b  )
-
\Lambda ^ {q , \lambda } (a )
 }{s}
 - \Xi ^ {a , b }
  \RN  _{W ^{1 , \infty } (0,T ; \R  ^N )}
&= 
\frac{1}{|s|}  \LN 
\hat{x} 
 \RN  _{W ^{1 , \infty } (0,T ; \R  ^N )} \\
& \leq 
 C_{4} |s|  \to 0 ~~(s \to 0 ) ,
\end{split}
\end{equation}
whence follows the assertion in Theorem \ref{Gateaux-solution}.
\end{proof}

By the uniform convergence \eqref{Apriori-Sec3-04}, we can immediately guarantee the following: 
\begin{Co}
\label{Gateaux-cost}
Let $ p ,q >3$. Then 
$J ^{q , \lambda } : \mathscr{C} \to \R $ is G\^{a}teaux differentiable
at any point and in any direction.
Moreover,  the G\^{a}teaux derivative 
${\rm d} J (a ; b )$ at  $a \in \mathscr{C} $  in the direction  $b  \in \mathscr{C} $ 
coincides with 
\begin{align*}
{\rm d} J ^ {q , \lambda } (a  ; b )  
&=
 \int_{0}^{T}  (   \Lambda ^ {q , \lambda } (a) (t) - x _{\ast} (t) ) \cdot \Xi ^ {a , b} (t)  dt 
+
 \int_{0}^{T}  a (t) \cdot b  (t)  dt  \\
&~~~~~~~~~
 + 
 \lambda  (  \Lambda ^ {q , \lambda } (a) (T) - z _ {\ast}  ) \cdot \Xi ^ {a , b} (T)  
+
\lambda   a (0 ) \cdot b  (0) ,
\end{align*}
where $  \Xi ^ {a, b } \in C ^1 ([0,T ] ; \R ^N )$ is the solution to \eqref{LinearizedProb}.
\end{Co}

We here define the following adjoint problem of \eqref{Approx-Prob}:
\begin{equation}
\begin{cases}
~\DS  - \gamma ' (t) +  D ^2 \varphi _{G , p ,  q } ( x ( t ) )  \gamma  (t) + \frac{1}{\lambda } H \gamma (t)  \\[2mm]
\DS \hspace{5cm} 
		=-  \frac{1}{\lambda } \LC x(t) - x_{\ast} (t) \RC , ~&~ t\in (0, T ), \\[2mm]  
~\gamma (T) = -   \LC x(T) - z_ {\ast} \RC,  \\[2mm]
~ x= \Lambda ^{q , \lambda } (a ),
\end{cases}
\tag{D$_{a, q , \lambda  }  $}
\label{DualProblem}
\end{equation}
where $x ^\ast $ and $z ^{\ast}$ are given target in \eqref{Cost-Appro}.
Note that the Hesse matrix $ D ^2 \varphi _{G , p ,  q } ( x ( t ) )  $ is symmetric.
Obviously, this problem possesses 
a unique solution $\gamma \in W ^{1, 2 } (  0,T ; \R ^N  )$ for any $a \in \mathscr{C}$
(recall that $ x _{\ast} \in L^2 (0,T ; \R ^N )$).
We define 
$ \Lambda ^ {q , \lambda } _{\ast} : \mathscr{C} \to W ^{1, 2 } (  0,T ; \R ^N  )$ by 
the relationship $ \Lambda ^ {q , \lambda } _{\ast} (a) = \gamma $. 
Then we can state the necessary optimality condition for the approximation problem
as follows:
\begin{Th}
\label{Opt-Condition} 
Let  $ a ^{\star ; q , \lambda  }  \in \mathscr{U}  ^ M  _{\rm ad} $ be a optimal control of $J ^ {q , \lambda }$
in $ \mathscr{U}  ^ M  _{\rm ad}  $.
Then $ \gamma ^{\star ; q , \lambda  }  := \Lambda _{\ast } ^ {q , \lambda   } (a ^{\star ; q , \lambda  }  ) $ satisfies 
$  a ^{\star ; q , \lambda  }=  H \gamma  ^{\star ; q , \lambda  }  $.
\end{Th}

\begin{proof}
Since $ \mathscr{U} _{\rm ad} ^  M$ is a convex set, 
$ \tau b  + (1 -\tau ) a ^{\star ; q , \lambda  }  \in \mathscr{U} ^\text{ad} _ M $ for any $ \tau \in ( 0, 1 ) $ and 
$b  \in  \mathscr{U} _{\rm ad} ^ M$.
Then passing the limit of 
\begin{equation*}
\frac{1}{\tau } \LC
J ^ {q , \lambda }
(\tau b  + (1 -\tau ) a ^{\star ; q , \lambda  } ) - J ( a ^{\star ; q , \lambda  }) \RC \geq 0
\end{equation*}
as $\tau \to 0 $, we get 
\begin{equation}
\label{opt-cond-02}
 {\rm d}  J ( a ^{\star ; q , \lambda  } ; b -  a ^{\star ; q , \lambda  }  )\geq  0 ~~~~\forall  b  \in  \mathscr{U} _{\rm ad} ^ M  .
\end{equation}

Here, by using $ \Lambda ^ {q , \lambda } _{\ast} (a) =  \gamma $ and \eqref{LinearizedProb}, 
we can rewrite $\text{d} J ^ {q , \lambda } (a  ; b ) $ as 
\begin{align*}
&\text{d} J ^ {q , \lambda } (a  ; b )  \\
=&
\lambda \int_{0}^{T} \LC    \gamma ' (t) -  D ^2 \varphi _{G , p ,  q } ( x ( t ) )  \gamma  (t) - \frac{1}{\lambda } H \gamma (t) \RC
 \cdot \Xi ^ {a , b} (t)  dt \\
&~~~~~~+
 \int_{0}^{T}  a (t) \cdot b (t)  dt  
 - 
 \lambda   \gamma (T) \cdot \Xi ^ {a , b} (T)  
+
\lambda   a (0 ) \cdot b  (0) \\
=&
- \lambda \int_{0}^{T}  \LC   \frac{d}{dt} \Xi ^ {a , b} (t) +  D ^2 \varphi _{G , p ,  q } ( x ( t ) ) \Xi ^ {a , b} (t)+ \frac{1}{\lambda } H \Xi ^ {a , b} (t) \RC
 \cdot \gamma  (t)  dt \\
&~~~~~~+
 \int_{0}^{T}  a (t) \cdot b (t)  dt  
 - 
 \lambda   \gamma (0) \cdot \Xi ^ {a , b} (0)  
+
\lambda   a (0 ) \cdot b  (0) \\
=&
 \int_{0}^{T}  ( a (t) -\gamma (t) )  \cdot b (t)  dt  
 +
\lambda   ( a (0 ) - \gamma (0) ) \cdot b  (0) .
\end{align*}
By replacing $b$ with $b -  a ^{\star ; q , \lambda  }  $ and using \eqref{opt-cond-02}, 
we can see that if $a ^{\star ; q , \lambda  } \in \mathscr{U} _{\rm ad} ^ M$
is an optimal control, then 
$ \gamma ^{\star ; q , \lambda  }  := \Lambda _{\ast } ^ {q , \lambda   } (a ^{\star ; q , \lambda  }  ) $ satisfies 
\begin{equation}
\begin{split}
\label{opt-cond-01} 
& \int_{0}^{T}  \LC  a ^{\star ; q , \lambda  }(t) -\gamma ^{\star ; q , \lambda  }   (t) \RC  \cdot \LC  b ( t)-   a ^{\star ; q , \lambda  } (t) \RC   dt  
\\
&\hspace{1cm} +
\lambda \LC   a ^{\star ; q , \lambda  } (0 )   -\gamma ^{\star ; q , \lambda  }   (0) \RC  \cdot \LC b ( 0)
	-   a ^{\star ; q , \lambda  } (0) \RC  \geq 0 \hspace{5mm } \forall b  \in  \mathscr{U} _{\rm ad} ^ M .
\end{split}
\end{equation}
According to the definition of $ \mathscr{U} _{\rm ad} ^ M$, 
there is some suitable small $\varepsilon > 0 $ such that $ \varepsilon  H \gamma ^{\star ; q , \lambda  } \in \mathscr{U} _{\rm ad} ^ M $.
Since $\mathscr{U} _{\rm ad} ^ M$ is convex,
$b =  \varepsilon  H \gamma ^{\star ; q , \lambda  } + (1- \varepsilon )a ^{\star ; q , \lambda  } \in \mathscr{U} _{\rm ad} ^ M $ 
holds. Therefore from \eqref{opt-cond-01} and $z \cdot b(t)  = Hz  \cdot b(t) $ for any $z \in \R ^N $ and $b \in \mathscr{C}$, 
we can derive 
\begin{equation*}
- \varepsilon  \int_{0}^{T}  \LN  a ^{\star ; q , \lambda  }(t) - H \gamma ^{\star ; q , \lambda  }   (t) \RN ^2 dt  
 -
\varepsilon \lambda \LN    a ^{\star ; q , \lambda  } (0 )   -\gamma ^{\star ; q , \lambda  }   (0) \RN ^2 \geq 0  
\end{equation*}
which yields $  a ^{\star ; q , \lambda  }=  H \gamma  ^{\star ; q , \lambda  }  $.
\end{proof}


\section{Convergence to the Original Problem}

In this final section, we discuss  the convergence of optimal controls
for the approximation problem. 
\begin{Th}
\label{Conv-OC-appro-to-Ori} 
Let $p >1 $ and fix $M > 0 $. Then regardless of the choice of  optimal controls
$ a ^{\star ; q , \lambda } $ of  $J  ^{q , \lambda }$ in $\mathscr{U} ^M _{\rm ad }$
for each parameter $q  > 1$ and $ 0 < \lambda <1  $,
the sequence
 $\{ a ^{\star ; q , \lambda } \} _{q > 1 , 0 < \lambda <1  } $
 is uniformly bounded in $W ^{1,2 } (0,T ; \R^ N )$ and 
possesses a subsequence
 $\{ a ^{\star ; q ^i  , \lambda ^ i } \} _{ i \in \N } $ ($q _i \to \infty , \lambda _i \to 0$) 
which converges to an optimal control of the original cost function $J$
as $i \to \infty $.
\end{Th}

\begin{proof}
Fix $ a \in \mathscr{C}  $ arbitrarily.
 Then
 for every $ q > 1 $ and $0< \lambda  < 1 $,
 \begin{align*}
J ^{q , \lambda } (  a ^{\star ; q , \lambda } ) 
&\leq  
J ^{q , \lambda } (  a ) \\
&\leq   
\frac{1}{2}
\int_{0}^{T} \|  \Lambda ^{q , \lambda }  ( a )  (t) - x _{\ast} (t) \| ^2 dt 
+  
\frac{1}{2}
\int_{0}^{T} \|  a  (t) \| ^2 dt  \\
&\hspace{4cm} + 
 \frac{1 }{2} \|  \Lambda ^{q , \lambda }  (a ) (T) - z _ {\ast}  \|  ^2  
+
 \frac{1 }{2}  \|  a (0 ) \|  ^2 
\end{align*}
holds regardless of the choice of optimal control $ a ^{\star ; q , \lambda } $.
By repeating the same argument in our proof of Theorem \ref{Convergence-to-original-sol},
the solutions 
$ \{ \Lambda ^{q , \lambda }  (a )  \} _{q > 1 , 0 < \lambda <1 }$
is uniformly bounded in $W ^{1,  2 } ( 0,T ; \R ^N )$ with respect to the parameters $q , \lambda $.
Hence $ \{   a ^{\star ; q , \lambda }    \} _{q > 1 , 0 < \lambda <1 }$
is also uniformly bounded in $W ^{1,2} (0, T ;\R ^N  )$ independently of  $q , \lambda $ 
by the definition of $J   ^{ q , \lambda } $ and $ \mathscr{U} ^M _{{\rm ad}} $.

Let $ \{   a ^{\star ; q ^i  , \lambda ^i }    \} _{i \in \N  }$ ($q ^ i  \to \infty , \lambda ^i \to 0 $)
be a convergent subsequence and $ a ^{\star  } \in \mathscr{U} ^M _{{\rm ad}}  $ be its limit.
Furthermore choose an optimal control of the original problem $a  ^{\star \star } \in \mathscr{U} ^M _{{\rm ad}} $.
Then by virtue of Theorem \ref{Convergence-to-original-sol},
$ \{ \Lambda ^{q ^i  , \lambda ^i }  ( a ^{\star ; q ^i  , \lambda ^i } )  \} _{ i \in \N  }$
and 
$ \{ \Lambda ^{q ^i  , \lambda ^i }  ( a ^{\star \star } )  \} _{ i \in \N  }$
converge to 
$ \Lambda   ( a ^{\star } )  $
and 
$  \Lambda   ( a ^{\star \star } ) $
strongly in $C([0,T] ; \R ^N )$,
respectively.
Taking the limit of 
\begin{equation*}
J   ^{ q ^i , \lambda ^i } ( a ^{\star ; q ^i  , \lambda ^i }  )
\leq 
J   ^{ q ^i , \lambda ^i } ( a ^{\star \star } )
\end{equation*}
as $i \to \infty $, we obtain $ J  ( a ^{\star  }  )\leq J ( a ^{\star \star } )$, 
which implies that the limit $ a ^{\star  }$ is also the minimizer of $J $.
\end{proof}

Finally,  we investigate what dose the limit of the optimality condition mean. 
\begin{Co}
\label{Conv-OC-appro-to-Ori-1} 
Fix $M >0 $ and $p > 3$.  Let 
$a ^{\star ; q ^i , \lambda ^i }$ be an optimal control of $ J  ^ {q ^i , \lambda ^i } $,   
$\{ a ^{\star ; q ^i , \lambda ^i } \} _{i \in \N  } \subset  \mathscr{U} _{\rm ad} ^ M$
be a convergent sequence in $C ([ 0, T ]; \R ^N )$
as $  q ^i \to \infty $ and $ \lambda ^i  \to 0$,
and let $ a ^{\star }  \in \mathscr{U} ^\text{ad} _ M$ be its limit.
Then for the solution to the adjoint  problem $ \gamma  ^{\star ; q ^i , \lambda ^i }  = \Lambda_{\ast } ^ {q ^i  , \lambda ^i  }
( a ^{\star ; q ^i , \lambda ^i } )$,
the sequence $\{ H \gamma  ^{\star ; q ^i , \lambda ^i }  \} _{i \in \N }$  
strongly converges to $a^{\star }  $ in $C ([ 0, T ]; \R ^N )$. 
\end{Co}

\begin{proof}
By Theorem \ref{Opt-Condition}, we have 
$H \gamma  ^{\star ; q ^i , \lambda ^i } = a ^{\star ; q ^i , \lambda ^i } $ for every $i \in \N $.
Hence by the assumption, we immediately see that 
$\{ H \gamma  ^{\star ; q ^i , \lambda ^i } \}_{i \in \N }$ 
strongly converges to the limit of  $\{  a ^{\star ; q ^i , \lambda ^i } \} _{i \in \N } $.
\end{proof}

\begin{Rem}
{\rm 
We see that the dummy target $z _ {\ast}$ does not need to match 
the genuine target $x_ \ast $ in our argument above. 
Although there might be  some advantage in computation of optimal control by  setting $z_{\ast} = x_{\ast } (T)$,
we unfortunately cannot find any  theoretical reason for it in this paper. 
}
\end{Rem}

\begin{Rem}
{\rm 
We here comment on the boundedness of $(\id - H ) \gamma  ^{\star ; q ^i , \lambda ^i } 
= (\gamma  ^{\star ; q ^i , \lambda ^i } _1 ,\ldots , \gamma  ^{\star ; q ^i , \lambda ^i }  _ n , 0 \ldots , 0)$.
Abbreviate $\gamma  ^{\star ; q ^i , \lambda ^i } $  to $\gamma  ^{ i } $.  
Recall that $\gamma  ^i $ satisfies 
\begin{equation}
\label{OC-convergence-01} 
\begin{cases}
\DS ~~ - \lambda ^i   \frac{d}{dt} \gamma  ^i  (t) +\lambda ^i   D ^2 \varphi _{G , p ,  q^i } ( x^i ( t ) ) \gamma  ^i  (t)
 + H \gamma  ^i (t) \\[2mm]
\hspace{5cm} \DS = -  \LC x ^ i(t) - x_ {\ast} (t) \RC ,  \\[2mm]
\DS ~~  \gamma ^i  (T) = -   \LC x ^i (T) - z_ {\ast} \RC, 
\end{cases}
\end{equation}
where $x ^ i  = \Lambda ^{q ^i , \lambda ^i } (a ^{\star ; q^ i , \lambda  ^i })$ is uniformly bounded
with respect to $i$.
Multiplying 
\eqref{OC-convergence-01} by $\gamma ^i $, we have 
\begin{equation*}
 - \frac{\lambda ^i}{2}\frac{d}{dt}  \| \gamma  ^i  (t) \| ^2  \leq   \|  x ^ i(t) - x_ {\ast} (t) \| \| \gamma ^i (t) \| ,
\end{equation*}
where we use the fact that 
$ D ^2 \varphi _{G , p ,  q ^i  } ( x ^i ( t ) )$ is non-negative and $H $ is linear monotone mapping.
Hence we get by the Gronwall inequality 
\begin{equation}
\label{OC-convergence-02} 
\sup _{i \in \N } \LC \sup_{ 0 \leq t \leq T} \lambda ^i  \| \gamma  ^i  (t) \| \RC < \infty .
\end{equation}
However,  it seems to be difficult to derive better estimates for $(\id - H ) \gamma  ^{\star ; q ^i , \lambda ^i }$
Indeed, by differentiating \eqref{Formula-Clique} once again, we obtain 
\begin{equation*}
\sup _{0\leq t \leq T } |\lambda ^i   D ^2 \varphi _{G , p ,  q ^i  } ( x^i ( t ) )  |
\leq 
C \lambda ^{i} (1 + q ^{i} )  \sup _{0\leq t \leq T }\| x^i ( t )  \| ^{p-2} 
\end{equation*}
with some constant $C$ independent of $\lambda ^i $ and $q ^i  $. 
Hence it is not easy to deal with the 2nd term of L.H.S. of \eqref{OC-convergence-01}
and show the boundedness as 
$\lambda ^i \to 0 $ and $q ^i \to \infty $. 
}
\end{Rem}


\appendix
\def\thesection{\Alph{section}}
\section{Appendix: Other Properties for the Clique Expansion of Hypergraph Laplacian}

In addition to the above mentioned in \S 2,
we can find some interesting properties 
of
the clique expansion of hypergraph Laplacian $D \varphi _{G ,p ,q  } $
as an approximation operator. 
On the remaining pages, we shall state some of them for future work of hypergraph  Laplacian.

We begin with the Poincar\'{e}-Wirtinger type inequality,
which the original hypergraph  Laplacian satisfies (see \cite{I-U}). 
We first prepare the following lemma.
\begin{Le}
\label{Lem-A1} 
Let $  p, q  >1$. 
Then it follows that 
\begin{equation}
\label{Ap-eq01} 
\bm{1} _V \cdot D \varphi _{G, p, q } (x) =0
~~~~
x \cdot D \varphi _{G, p, q } (x) =p \varphi _{G ,p ,q  } (x)
~~~~~~\forall x \in \R ^N ,
\end{equation}
where 
$\bm{1} _V :=  \sum_{i =1}^{N} \bm{1} _ i = (1 , \ldots, 1 ) $.
\end{Le}
\begin{proof}
Recall \eqref{Formula-Clique}, i.e., the specific formula of $\partial _{x_ k } \varphi _{G ,p ,q } (x)$.
Defining 
\begin{equation*}
a_{ij   }
:= 
\begin{cases}
~~ 0  ~&~ \text{ if } i = j  , \\[2mm]
~~\DS p
\sum_{e \in E \text{ s.t. } i , j   \in e }  w(e) 
 ( f_{e ,q } (x) ) ^{p -q  }
 |x_i  -x_ j | ^{ q-2 } (x_i   -x_ j)  
~&~ \text{ if } i \neq j  ,
\end{cases}
\end{equation*}
we can see that 
$ \partial _{x_i } \varphi _{G, p ,q  } (x)  = \sum_{j  =1}^{N }  a_{ij  } $
and 
$ \bm{1} _V \cdot D \varphi _{G, p ,q  } (x)  = \sum_{ i , j =1}^{N }  a_{ij  } $.
Since $a_{ij  } = - a _{ ji  }$ and $a_{ii} = 0 $, we obtain the first identity of \eqref{Ap-eq01}. 
Moreover,  we have 
\begin{align*}
x \cdot  D \varphi _{G ,p,q } (x )  
&= 
\sum_{i  =1 }^{N }x _i \partial _{ x_i } \varphi _{G, p ,q } (x) 
= \sum_{i , j  =1 }^{N } x _i   a _ { ij   } \\
&  =
\sum_{i < j }  ( x _i  a _ {ij  } + x_ j a_{ji } ) 
=
\sum_{i < j }  ( x _ i -  x_ j )  a_{ ij  }  \\
&  =
p
\sum_{i < j  } ~~
\sum_{e \in E \text{ s.t. } i , j  \in e }  w(e) 
 ( f_{e ,q } (x)) ^{p -q  }
 |x_i   -x_ j | ^{ q }  \\
&=
p
\sum_{e \in E }  ~~
w(e) 
 ( f_{e ,q } (x)) ^{p -q  }
\sum_{i< j ~~\text{s.t. } i , j  \in e }  
 |x_i   -x_ j | ^{ q }  \\
& = 
p
\sum_{e \in E }  ~~
w(e) 
 ( f_{e ,q } (x)) ^{p -q  }
 ( f_{e ,q } (x)) ^{q  }
 = 
p \varphi _{G ,p , q } (x ).
\end{align*}
Hence the second identity of  \eqref{Ap-eq01} also holds.
\end{proof}

To state the next assertion, we define the mean value of $ x\in \R ^N $ by 
\begin{equation*}
\overline{x} := \LC \frac{1}{N}  \sum_{i=1}^{N} x_i \RC \bm{1}_V 
=\LC 
\frac{1}{N}  \sum_{i=1}^{N} x_i , \ldots ,
\frac{1}{N}  \sum_{i=1}^{N} x_i \RC .
\end{equation*}
Moreover, we  assume that the hypergraph $G =(V ,E ,w)$ is {\it connected},
i.e., for every $ i , j  \in V $ there exist some
$ \mu _1 ,\ldots, \mu _{k-1}  \in V $ and $ e _1 , \ldots e_k \in E$ such that
$ \mu _ {l-1} , \mu  _ l  \in e_l$ holds for any $l=1, 2, \ldots , k $, 
where $\mu_ 0 = i $ and $ \mu _k = j $.
In this case we can define the diameter of the hypergraph by 
\begin{align*}
\text{diam}_G &:= \max _{  i, j  \in V}  \text{dist} (i,j  ) , \\
\text{dist} (i,j  ) &:= 
\min \LD k \in \N  ;~~~
\begin{matrix}
 \mu _1 ,\ldots, \mu _{k-1}  \in V , ~~~
e _1  , \ldots  , e_k \in E\\[2mm]
 ~~\text{s.t.}~~  \mu _ {l-1} , \mu  _ l  \in e_l~~ \forall j=1, \ldots , n , \\[2mm]
 ~~\text{where}~~ \mu _0  =i  , \mu _k  = j . 
\end{matrix}
\RD .
\end{align*}

\begin{Th}
\label{Poincare}
 Let $ p ,q > 1  $ and assume that $G$ is connected.
Then there exist constants $ \gamma _{G, p} , \Gamma _{G, p} > 0 $
which depend only on $p$ and $G$ and are independent of $q$
such that 
\begin{equation}
\label{Poincare-Clique} 
\gamma _{G, p } \|  x - \overline{x} \| ^p  \leq p \varphi _{G,p ,q  } (x) \leq 
\Gamma _{G, p  } \|  x - \overline{x} \| ^p 
~~~~~\forall x \in \R ^N .
\end{equation}
\end{Th}

\begin{proof}
Fix $ i , j \in V $ arbitrarily.
By the assumption, 
 there are
$ \mu _1 ,\ldots, \mu _{k-1}  \in V $ and $ e _1 , e _2 , \ldots e_k \in E$ such that
$ \mu _ {l-1} , \mu  _ l  \in e_l$ holds for any $l=1, 2, \ldots , k $.
From the subadditivity, the H\"{o}lder inequality, and 
$ | x_{\mu _{l} } -  x_{\mu _{l -1 } }  | \leq f_{e _{\mu _l} , q } (x)$, 
we can derive 
\begin{align*}
| x_ i  - x_ j  | 
&\leq 
\sum_{l =1}^{k }  |  x_{\mu _{l} } -  x_{\mu _{l -1 } }  | 
\leq \sum_{l =1}^{k }  f_{e _{\mu _l} , q }  (x) \\
& \leq \LC   \sum_{l =1}^{k }   ( f_{e _{\mu _l} , q }  (x) ) ^p \RC ^{1/p }  \text{dist} (i, j )   ^{1/ p'} \\
&\leq 
 \frac{ \text{diam} _G   ^{1/ p'}}{ ( \min _{e\in E }   w(e) ) ^{1/p }   } \LC  \sum_{e\in E } w(e) ( f_{e  , q } (x) ) ^p \RC ^{1/p }   .
\end{align*}
By the definition of $\overline{x} $, we get 
\begin{equation*}
|  \overline{x} _ i  - x _ i  |
  \leq \frac{1}{N} \sum_{j=1}^{N} |x_ i - x_ j | \leq 
 \frac{ \text{diam} _G   ^{1/ p'}}{  ( \min _{e\in E }   w(e) ) ^{1/p }   } \LC  \sum_{e\in E } w(e) ( f_{e  , q } (x) ) ^p \RC ^{1/p }   
\end{equation*}
and then we obtain the lower inequality of \eqref{Poincare-Clique} with
\begin{equation*}
\gamma _{G , p }
= \frac{   \min _{e\in E }   w(e)  }{ N ^p~ \text{diam} _G   ^{ p -1 }} .
\end{equation*}

On the other hand, we get by the triangle inequality for the $\ell  ^q$-norm
\begin{align*}
f_{e , q } (x)
 &= \LC \frac{1}{2}  \sum_{ i, j  \in e  } |x _ i  - x _ j  |  ^q  \RC ^{1/q} 
		= \LC \frac{1}{2}  \sum_{ i, j  \in e  } |x _ i  - \overline{x} _ i + \overline{x} _ j   - x _ j  |  ^q  \RC ^{1/q} \\
 & \leq  \LC \frac{1 }{2}  \sum_{ i, j  \in e  } \LC  |x _ i  - \overline{x} _ i | ^q + |  x _ j -  \overline{x} _ j     |  ^q  \RC  \RC ^{1/q} \\
&	\leq 
 \LC \# e  \RC ^{1/q}  \LC  \sum_{i=1}^{N }  |x _ i   -\overline{x} _ i  |  ^q  \RC ^{1/q} \\
&	
 \leq 
 \begin{cases}
~~\# e  ^{1/q} \| x   -\overline{x} \| ~~&~~\text{if}~~ q \geq 2 , \\
~~\# e  ^{1/q} N ^{( 2 -q ) / 2q   }\| x   -\overline{x} \|  ~~&~~\text{if}~~ q < 2 .
\end{cases}
\end{align*}
Hence with 
\begin{equation*}
\Gamma _{G, p, q}:= 
\begin{cases}
~~ \DS \LC \sum_{e\in E} w(e)   \# e  ^{p/q} \RC ~~&~~\text{if}~~ q \geq 2 ,  \\
~~ \DS \LC \sum_{e\in E} w(e)   \# e  ^{p/q}  \RC  N ^{ p (2 -q  ) / 2q   } ~~&~~\text{if}~~ q < 2 ,
\end{cases}
\end{equation*}
the inverse inequality of \eqref{Poincare-Clique} holds. 
Moreover, by setting 
\begin{equation*}
\Gamma  _{G , p } :=  \LC \sum_{e\in E} w(e)   \# e  ^{p}  \RC  N ^{ p /2   }
\end{equation*}
we have 
$\Gamma  _{G , p ,q } \leq \Gamma  _{G , p }$ for any $q \in [ 1 , \infty )$
and then $\Gamma _{G , p ,q }$ can be replaced with  $\Gamma _{G , p  }$.
\end{proof}

By applying this inequality to the Cauchy problem 
\begin{equation}
\label{decay-Clique}
\begin{cases}
~~ x' (t) + D \varphi _{G, p, q } (x(t)) = 0 ,~~&~~t > 0 ,\\
~~ x(0 ) =x_ 0 ,
\end{cases}
\end{equation}
we can deduce the same type decay estimate of solution 
as that for the Cauchy problem associated with the original hypergraph Laplacian (see \cite{I-U}).
\begin{Th}
\label{Decay}
Let $x $ be a solution to 
\eqref{decay-Clique} and 
define  $X (t) := \| x (t) - \overline{x_0} \| $.
Then for every $t \geq 0 $,
\begin{align*}
&~ X(t) \leq  \LC X(0) ^{2-p }  
			-  ( 2-p )  \gamma _{G,p }  t   \RC _+ ^{1 /  ( 2-p )  } 
& \text{if } 1 \leq p<2, \\[3mm]
 & ~X(t) \leq  X (0) \exp \LC  -  \gamma _{G,p } t \RC
& \text{if } p=2, \\[3mm]
&~ X(t) \leq \LC \frac{1}{X(0) ^{p-2 } }
			+( p-2 )  \gamma _{G,p }  t   \RC  ^{ - 1 /  ( p-2  )  }
& \text{if } p> 2, 
\end{align*}
and
\begin{align*}
&~ X(t)  \geq \LC X(0) ^{2-p }  
			- ( 2-p ) \Gamma _{G,p }   t   \RC _+ ^{1 /  ( 2-p )  }  
& \text{if } 1 \leq p<2, \\[3mm]
&~ X(t)  \geq X (0) \exp \LC -  \Gamma _{G,p }  t \RC  
& \text{if } p=2, \\[3mm]
&~ X(t)  \geq  \LC \frac{1}{X(0) ^{p-2 } } 
			+( p-2 ) \Gamma _{G,p }   t   \RC  ^{ - 1 /  ( p-2  )  }
& \text{if } p> 2, 
\end{align*}
where  $( s ) _+ := \max \{ s, 0\}$ and $ \Gamma _{G,p} ,\gamma _{G,p} $ are constants in Theorem \ref{Poincare}.
\end{Th}
\begin{proof}
Multiplying  \eqref{decay-Clique} by $\bm{1} _ V $ and using Lemma \ref{Lem-A1}, we have $ \overline{x} (t) '  = 0  $,
namely, $ \overline{x} (t)   =   \overline{x} _0  $ for every $t > 0 $. 
Then testing 
\begin{equation*}
( x (t)  - \overline{x} (t) ) '   + D \varphi _{G, p, q } (x(t)) = 0
\end{equation*}
by $x (t)  - \overline{x} (t) $, 
we have 
\begin{equation*}
 \frac{d}{dt} \|   x (t)  - \overline{x} (t) \| ^2    + 2p  \varphi _{G, p, q } (x(t)) = 0
\end{equation*}
and from Theorem \ref{Decay}, we derive
\begin{equation*}
 - 2   \Gamma _{G, p}   \|   x (t)  - \overline{x} (t) \| ^p 
\leq 
 \frac{d}{dt} \|   x (t)  - \overline{x} (t) \| ^2 
  \leq  - 2   \gamma _{G, p}   \|   x (t)  - \overline{x} (t) \| ^p .
\end{equation*}
Immediately, Theorem \ref{Decay} follows.
\end{proof}

\begin{Rem}
{\rm Let $  \varphi  ^\lambda _{G , p  }$ stand for the Moreau-Yosida regularization of $\varphi _{G ,p }$ 
and $ R ^ {\lambda } _{G , p } :=  ( { \rm id} +\lambda \partial \varphi _{G,p }) ^{-1} $ 
be the resolvent of the hypergraph Laplacian with $\lambda > 0 $.
Then we have 
\begin{align*}
\varphi   ^\lambda _{G , p  } (x) 
& \geq  \frac{1}{2\lambda } \| x - R ^ {\lambda } _{G , p } x \| ^2 
 + \varphi    _{G , p  } ( R ^ {\lambda } _{G , p } x ) \\
& \geq  \gamma _{G , p } \|  R ^ {\lambda } _{G , p } x - \overline{R ^ {\lambda } _{G , p } x } \| ^p 
\geq  \gamma _{G , p } \|  R ^ {\lambda } _{G , p } x - \overline{x } \| ^p. 
\end{align*}
Here remark that 
 we can  get  $\overline{R ^ {\lambda } _{G , p } x } = \overline{ x }  $ by
testing $ R ^ {\lambda } _{G , p }  x+ \lambda  \partial \varphi _{G ,p } (   R ^ {\lambda } _{G , p }  x) \ni x   $
by $\bm{1} _ V $. 
Hence though the Yosida approximation is 
one of the most standard approximations for the subdifferential and maximal monotone operator, 
it does not satisfy the Poincar\'{e} inequality,
which the original hypergraph Laplacian fulfills.
Hence  our approximation based on the clique expansion 
might be  better than other approximations in terms of the structural preserving. 
}
\end{Rem}

We next consider the convergence of the resolvent  and the Yosida approximation.
According to the abstract results \cite{Mosco1, Mosco2}, 
$( { \rm id} +\lambda \partial \varphi _{G,p , q }) ^{-1}  (x )$
converges to 
$( { \rm id} +\lambda \partial \varphi _{G,p }) ^{-1}  (x )$
since Lemma \ref{Clique-Conver} implies that $ \varphi _{G , p , q } \to \varphi _{G, p }$
in the sense of Mosco.
The resolvent of the hypergraph Laplacian 
is used to investigate the geometrical structure of  the hypergraphs in \cite{Aka, IKTU, KM}
and the PageRank of the hypergraph in \cite{TMIY}.
Hence by establishing more accurate convergence rate of the resolvent 
we might be able to study these problem more precisely 
with the clique expansion of the hypergraph Laplacian.
Henceforth, let 
the resolvent of $\partial \varphi _{G ,p }$ and   $D\varphi _{G ,p ,q }$  be denoted by 
\begin{equation*}
R ^ {\lambda } _{G , p } :=  ( { \rm id} +\lambda \partial \varphi _{G,p }) ^{-1} ,~~~~~
R ^ {\lambda } _{G , p ,q  } :=  ( { \rm id} +\lambda D \varphi _{G,p ,q  }) ^{-1} ,
\end{equation*}
and the Yosida approximation by 
\begin{equation*}
A ^ {\lambda } _{G , p } := \frac{{\rm id}  -  R ^ {\lambda } _{G , p }}{\lambda } ,~~~~~
A ^ {\lambda } _{G , p ,q  } := \frac{{\rm id}  -  R ^ {\lambda } _{G , p ,q }}{\lambda }  .
\end{equation*}

\begin{Th}
\label{Resolvent} 
For any $ p ,q  >1  $, 
the resolvent of $\partial \varphi _{G ,p }$ and $D \varphi _{G,p ,q  }$  satisfy
\begin{equation}
\label{Res-Conv}
\|  R ^ {\lambda } _{G , p } x - R ^ {\lambda } _{G , p , q} x  \| ^2    \leq 
\lambda  
\kappa _{G ,p }
\LC \nu _E ^{p/q } -1 
\RC  \| x \|  ^p ~~~~\forall x \in \R ^N ,
\end{equation}
where $\nu _E := \max _{e \in E} \LC \frac{\# e (\# e -1 )}{2 }  \RC  $ and 
$\kappa _{G ,p }$ is a constant  in \eqref{bound}.
Moreover, if 
\begin{equation}
\label{YosiAp-Conv01}
 \frac{ p \log \nu _E }{ \log \LC \lambda ^{1+\delta } +1   \RC } \leq q   
\end{equation}
with some $\delta > 0 $, then 
the Yosida approximation of $\partial \varphi _{G ,p }$ and $D \varphi _{G,p ,q  }$  satisfy
\begin{equation}
\label{YosiAp-Conv02}
\|  A ^ {\lambda } _{G , p } x - A ^ {\lambda } _{G , p , q} x  \| ^2    \leq 
\lambda  ^{\delta } \kappa _{G ,p } \| x \| ^p     ~~~~\forall x \in \R ^N .
\end{equation}
\end{Th}

\begin{proof}
Let $ \xi : = R ^ {\lambda } _{G , p }  x $ and $\xi ^{q } := R ^ {\lambda } _{G , p ,q }  x  $, i.e., these be solutions to 
\begin{equation*}
\xi + \lambda  \partial \varphi _{G ,p } ( \xi ) \ni x , ~~~~~~
\xi ^q + \lambda  D \varphi _{G ,p ,q } ( \xi  ^q ) = x  .
\end{equation*}
Multiplying 
\begin{equation*}
( \xi ^q  - \xi )   + \lambda  D \varphi _{G ,p ,q } ( \xi  ^q )
- \lambda  \partial \varphi _{G ,p } ( \xi ) \ni 0 
\end{equation*}
by $ \xi ^q  - \xi$ and using 
\begin{align*}
& D \varphi _{G ,p ,q } ( \xi  ^q ) \cdot  (\xi ^q  - \xi) \geq \varphi _{G ,p ,q } ( \xi  ^q ) - \varphi _{G ,p ,q } ( \xi  ) , \\
& \eta  \cdot (\xi ^q  - \xi) \leq  \varphi _{G ,p } (\xi ^q ) - \varphi _{G ,p } (\xi  ) 
~~~~~ 
(\eta \in \partial \varphi _{G ,p } ( \xi ) ) ,
\end{align*}
we obtain 
\begin{equation*}
\|  \xi ^q  - \xi  \| ^2    + \lambda  
\LC 
\varphi _{G ,p ,q } ( \xi  ^q ) - \varphi _{G ,p ,q } ( \xi  )
- \varphi _{G ,p } (\xi ^q ) +  \varphi _{G ,p } (\xi  ) 
\RC 
\leq  0 .
\end{equation*}
Since $\varphi _{G ,p ,q } ( \xi  ^q )  \geq \varphi _{G ,p } (\xi ^q )$,
we derive \eqref{Res-Conv} from Lemma \ref{Clique-Conver}  and \eqref{bound}.

When $ \nu _E  ^{p /q} -1 \leq \lambda ^{1 + \delta }$, which is equivalent to \eqref{YosiAp-Conv01},
we have 
\begin{equation*}
\|  x-  R ^ {\lambda } _{G , p } x - x +  R ^ {\lambda } _{G , p , q} x  \| ^2    \leq 
\lambda  ^{2 +\delta }
\kappa _{G ,p }
  \| x \|  ^p .
\end{equation*}
Hence dividing this inequality by $\lambda ^2 $, we obtain \eqref{YosiAp-Conv02}.
\end{proof}

\begin{Rem}
{\rm
It is well known that  
the Yosida approximation of the maximal monotone operator 
converges to the minimal section. Namely, $A ^{\lambda } _{G , p } x $ tends to 
\begin{equation*}
( \partial \varphi _{G,p } ) ^{\circ } (x)
:= \LD \eta ^\circ \in   \partial \varphi _{G,p }  (x) ; ~
\| \eta ^ \circ \|  = \min _{ \eta \in   \partial \varphi _{G,p }  (x) } \| \eta \| \RD  ,
\end{equation*}
where such $\eta ^{\circ } $ is determined uniquely since 
$ \partial \varphi _{G,p }  (x) $ forms a  closed convex subset in $\R ^N $.
Hence \eqref{YosiAp-Conv02} implies that 
the Yosida approximation of $D  \varphi _{G,p ,q  }$ (approximation based on the clique expansion)
converges to 
$( \partial \varphi _{G,p } ) ^{\circ } (x) $ as $ \lambda \to 0 $ and $q\to \infty $ appropriately.
}
\end{Rem}

Finally, we consider the first positive eigenvalue of the hypergraph Laplacian and its approximation.
It is easy to see that 
\begin{equation*}
\varphi _{G ,p } (x) = 0 ~~\Leftrightarrow 
~~\varphi _{G ,p , q } (x) = 0 ~~\Leftrightarrow ~~
\exists c \in \R ~~\text{s.t.}~~x = c \bm{1} _V = (c , \ldots,  c )
\end{equation*}
if  $G $ is connected.
This implies that 
\begin{equation*}
\partial \varphi _{G ,p } (x) \ni 0 ~~\Leftrightarrow 
~~D \varphi _{G ,p , q } (x) = 0 ~~\Leftrightarrow ~~
\exists c \in \R ~~\text{s.t.}~~x = c \bm{1} _V = (c , \ldots,  c )
\end{equation*}
and $ 0 $ is first eigenvalue of $ \partial \varphi _{G ,p } $ and $D \varphi _{G ,p , q }$.
Then we set 
the orthogonal complement of first eigenspace $\{  x = c\bm{1} _V = (c,\ldots , c ) \in \R ^V ; c \in R \}$
by 
\begin{equation*}
\R ^N _{0} := \{ x \in \R ^N ; ~ x \cdot \bm{1} _V = 0 \}
= \{ x \in \R ^N ; ~ \overline{ x } = 0 \}
\end{equation*}
and define 
the second (i.e., the first positive) eigenvalue by 
\begin{equation*}
\lambda  _{1} :=  \inf  _{x \in \R ^N _ 0 \setminus \{ 0\} } \frac{ p \varphi _{G, p  } (x)}{\| x \| ^p} ,
~~~~~
\lambda _{1 ,q } :=  \inf  _{x \in \R ^N _ 0 \setminus \{ 0\} } \frac{ p \varphi _{G, p ,q } (x)}{\| x \| ^p} .
\end{equation*}
By Theorem  \ref{Poincare} and Lemma \ref{Clique-Conver},
$\lambda  _{1} $ and $\lambda _{1, q} $ are bounded by $ \gamma _{G,p}  $ and $\Gamma _{G,p }  $ from below and above, respectively. 
Hence we can show the following.
\begin{Le}
\label{Lem-eigen}
Let $ p,q > 1$ and $G $ be connected. Then the mappings 
 $ x \mapsto   p \varphi _{G, p  } (x) / \| x \| ^p $ and 
 $ x \mapsto   p \varphi _{G, p ,q  } (x) / \| x \| ^p $ attain their minimum on 
$\R ^N _ 0 \setminus \{ 0\} $
at some $  \zeta _1 , \zeta  _{1 ,q } \in \R ^N _ 0 $, respectively.
Moreover, 
\begin{equation}
\label{eigen-01} 
\lambda  _{1} =   \frac{ p \varphi _{G, p  } (\zeta _1  )}{\| \zeta _1  \| ^p} ,
~~~~~
\lambda _{1 ,q } =   \frac{ p \varphi _{G, p ,q } (\zeta  _{1 , q }  )}{\| \zeta  _{1 , q }  \| ^p} 
\end{equation}
holds with some  $  \zeta _1 , \zeta  _{1 ,q } \in \R ^N _ 0 $ if and only if 
these satisfy 
\begin{equation}
\label{eigen-02}
\lambda  _{1} \|  \zeta _1  \| ^{q-2} \zeta _1\in \partial  \varphi _{G, p  } (\zeta _1  ) , 
~~~~~
\lambda  _{1, q} \| \zeta  _{1 ,q }  \| ^{q-2}  \zeta  _{1 ,q }  =  D  \varphi _{G, p ,q   } (\zeta  _{1 ,q }  ) .
\end{equation}

\end{Le}

\begin{proof}
Since 
\begin{equation*}
f_{e} (c x ) = |c| f_e (x), ~~~~f_{e , q} (c x ) = |c| f_{ e , q } (x)
\end{equation*}
hold for any $x \in \R ^N $ and $ c \in \R $,
we can easily see that $\varphi _{G ,p }$ and $\varphi _{G ,p , q}$ are  
homogeneous of degree $p$ 
and we can restrict ourselves to $x  $ belonging to 
a compact set $ \BS ^N _ 0 := \{ x \in \R ^N _0 ;  ~\| x\| = 1 \} $.
Then by the standard argument for the convergence of minimizing sequences,
we can assure that 
 $ x \mapsto   p \varphi _{G, p  } (x) / \| x \| ^p $ and 
 $ x \mapsto   p \varphi _{G, p ,q  } (x) / \| x \| ^p $ attain their minimum at some $  \zeta _1 , \zeta  _{1 ,q } \in \BS ^N _ 0 $.

Assume that $  \zeta _{1 ,q }  \in \R  ^N _ 0 \setminus \{ 0 \}$ satisfies \eqref{eigen-01}.
Since $x \mapsto  \| x\| ^ p $ is convex and its derivative is $ p \| x \| ^{p-2} x$,
we can get by the definition of the subdifferential 
\begin{equation*}
\| z \| ^p -  \| \zeta _{1 ,q }  \| ^p \geq p \| \zeta _{1 ,q } \| ^{p-2} \zeta _{1 ,q } \cdot (z - \zeta _{1 ,q }) ~~~\forall z \in \R ^N . 
\end{equation*}
By the definition of $\lambda _{1 ,q } $, we have $ \lambda _{1 ,q } \leq p \varphi _{G, p ,q  } (x) / \|  x\| ^p  $ for any $x \in \R ^N _0 \setminus \{ 0 \}$
and 
\begin{equation*}
\frac{ p \varphi _{G, p ,q } (x )}{\lambda _{1 ,q } } -
\frac{ p \varphi _{G, p ,q } (\zeta _{1 ,q } )}{\lambda _{1 ,q } }  \geq p \| \zeta _{1 ,q } \| ^{p-2} \zeta _{1 ,q } \cdot (x - \zeta _{1 ,q })
 ~~~\forall x \in \R ^N _0. 
\end{equation*}
Here for every $z \in \R ^N $, there exist $x \in \R ^N _ 0$ and $ c \in \R $ such that $ z = x + c \bm{1} _V$.
Moreover, from $ f_{e ,q } ( x + c \bm{1} _V  ) = f _{e ,q } ( x ) $, 
 $ \varphi _{G ,p ,q } ( x + c \bm{1} _V  ) = \varphi _{G ,p ,q } ( x ) $,
and $\zeta _{1 ,q } \cdot (x + c \bm{1} _V ) = \zeta _{1 ,q } \cdot x $, it follows that 
\begin{equation*}
\varphi _{G, p } (z ) -
 \varphi _{G, p } (\zeta _{1 ,q } ) \geq  \lambda _{1 ,q } \| \zeta _{1 ,q } \| ^{p-2} \zeta_{1 ,q } \cdot (z - \zeta _{1 ,q })
 ~~~\forall z \in \R ^N , 
\end{equation*}
which implies that $ \lambda _{1 ,q } \| \zeta _{1 ,q } \| ^{p-2} \zeta _{1 ,q } $ satisfies the definition of the subgradient 
of $\varphi _{G,p ,q } $ at $\zeta _{1 ,q } $.
By exactly the same argument,  
$  \zeta _{1  }  \in \R  ^N _ 0 \setminus \{ 0 \}$ satisfying  \eqref{eigen-01} also fulfills \eqref{eigen-02}

Conversely, if 
 $  \zeta _{1 ,q }  \in \R  ^N  \setminus \{ 0 \}$ satisfies \eqref{eigen-02}, 
we have  $  \zeta _{1 ,q }\cdot \bm{1}_V   = 0$, i.e., $  \zeta _{1 ,q }\in \R ^N _0 $.
Then multiplying \eqref{eigen-02}  by $\zeta _{1 ,q }$ and using  \eqref{Ap-eq01}, we obtain  
\begin{equation*}
 \lambda _{1 ,q } \| \zeta _{1 ,q } \| ^{p} =p \varphi _{G, p , q} (\zeta _{1 ,q } ) , 
\end{equation*}
which implies that $ \zeta _{1 , q }$ is the minimizer of  $ x \mapsto   p \varphi _{G, p ,q  } (x) / \| x \| ^p $
on $\R ^N _0$. 
Since the original hypergraph also satisfies the same type identity as  \eqref{Ap-eq01}
(see \cite{I-U}), 
we can apply this argument to 
 $  \zeta _{1 }  \in \R  ^N  \setminus \{ 0 \}$ satisfying  \eqref{eigen-02}
 and 
 we can show that 
$  \zeta _{1 }  \in \R  ^N  \setminus \{ 0 \}$ 
 is the minimizer of  $ x \mapsto   p \varphi _{G, p  } (x) / \| x \| ^p $
on $\R ^N _0$.
\end{proof}

We can show $\lambda _{1 , q } \to \lambda _1 $ as $q \to \infty $. On the other hand, 
since $\varphi _{G,p } (- z)  =\varphi _{G,p } (z)  $ and  $\varphi _{G,p ,q } (- z)  =\varphi _{G,p ,q} (z)  $ hold,
the minimizers of 
 $ x \mapsto   p \varphi _{G, p  } (x) / \| x \| ^p $, $ x \mapsto   p \varphi _{G, p ,q  } (x) / \| x \| ^p $ are not determined uniquely
 in general
 and then $ \zeta _{1 , q } \to \zeta _1$ is not necessarily satisfied depending on the choice of 
$\zeta _1 $ and the sequence $\{ \zeta _{1, q} \}$.
However, we can assure the following:

\begin{Th}
Let $\lambda _1 $ and $\lambda _{1, q } $ ($q >1 $) be defined by \eqref{eigen-01}.
Then 
\begin{equation}
\label{eigen-03}
\lambda _1  \leq \lambda _{1 ,q } \leq \nu ^{p /q } _E \lambda _1 , 
\end{equation}
where $\nu _E := \max _{e \in E} \LC \frac{\# e (\# e -1 )}{2 }  \RC  $.
Furthermore, 
let $\{  \zeta _{1 , q } \} _{q >1 } \subset \BS ^{N }_ 0$ be a sequence of solutions to \eqref{eigen-02}.
Then for any convergent  subsequence 
$ \{  \zeta _{1 , q ^j } \} _{j \in \N } \subset \{  \zeta _{1 , q } \} _{q >1 }  $
($q _ j \to \infty$ as $j \to \infty $), its limit $\zeta \in \BS ^N _0 $ satisfy 
\begin{equation}
\label{eigen-04}
\lambda  _{1} =    p \varphi _{G, p  } (\zeta  ) ,
~~~~
\text{that is to say,}
~~~~
\lambda  _{1}  \zeta \in \partial  \varphi _{G, p  } (\zeta  ) .
\end{equation}
\end{Th}

\begin{proof}
According to \eqref{Sec2-6-1}, we have 
\begin{align*}
& \lambda _1 = \varphi _{G ,p } (\zeta _{1}) \leq \varphi _{G ,p } (\zeta _{1, q }) 
\leq 
\varphi _{G ,p , q  } (\zeta _{1, q }) = \lambda _{1,q} ,
\\
& \lambda _{1, q } =  \varphi _{G ,p , q  } (\zeta _{1, q } ) \leq 
   \varphi _{G ,p ,q  } (\zeta _{1})
\leq 
  \nu _E  ^{ \frac{p}{q }    }  \varphi _{G ,p  } (\zeta _{1}) =  \nu _E  ^{ \frac{p}{q }    } \lambda _1 ,  
\end{align*}
which leads to \eqref{eigen-03}.
Next let  
$ \{  \zeta _{1 , q ^j } \} _{j \in \N } $ be a convergent subsequence of $ \{  \zeta _{1 , q } \} _{q >1 }  $
and its limit be written by  $\zeta \in \BS ^N _0 $.
By \eqref{eigen-02} and the definition of the subgradient, we have
\begin{equation*}
\lambda _{1 , q ^j } \zeta _{1 , q ^j } (z - \zeta _{1 , q ^j } ) \leq \varphi _{G, p ,q^j  } (z) -\varphi _{G, p ,q^j  } ( \zeta _{1 , q ^j }) .
\end{equation*}
Since from Lemma \ref{Clique-Conver}, i.e.,  the uniform convergence of $\varphi _{G ,p, q ^j }$
to $\varphi _{G, p} $ on the compact sets $\BS ^N _ 0 $, we obtain 
\begin{equation*}
\lambda _{1  } \zeta  (z - \zeta  ) \leq \varphi _{G, p   } (z) -\varphi _{G, p } ( \zeta  ) ,
\end{equation*}
which implies $ \lambda _{1  } \zeta  \in \partial \varphi _{G, p } ( \zeta  )$. 
\end{proof}

\subsection*{Acknowedgements}
Takeshi Fukao  is supported by 
JSPS Grant-in-Aid for Scientific Research (C) (No.21K03309). 
Masahiro Ikeda  is supported by 
JSPS Grant-in-Aid for Scientific Research (C) (No.23K03174)
and 
JST CREST Grant (No.JPMJCR\\1913).
Shun Uchida is supported by 
JSPS Grant-in-Aid for Scientific Research (C) (No.24K06799)
and 
Sumitomo Foundation Fiscal 2022 Grant for Basic Science Research Projects (No.2200250).


\begin{thebibliography}{00}



\bibitem{Aka}
T, Akamatsu, 
A new transport distance and its associated Ricci curvature of hypergraphs,
Anal. Geom. Metr. Spaces 10(1) (2022), 90--108.


\bibitem{Mosco1}
H. Attouch, Convergence de fonctionnelles convexes,
Journees d'Analyse Non Lineaire (Proc. Conf., Besancon, 1977), 1--40,
in Lecture Notes in Math., 665, Springer, Berlin, 1978. 


\bibitem{Mosco2}
H. Attouch, 
Familles d'operateurs maximaux monotones et mesurabilit\'{e},
Ann. Mat. Pura Appl. (4) 120 (1979), 35--111.




\bibitem{Bar}
V. Barbu,
{\it Nonlinear Differential Equations of Monotone Types in Banach Spaces},
Springer, New York, 2010.

\bibitem{PUni}
H. Br\'{e}zis, Probl\`{e}mes unilat\'{e}raux,
J. Math. Pures Appl. 51(9) (1972), 1--168. 


\bibitem{Bre}
H. Br\'{e}zis, 
Op\'{e}rateurs maximaux monotones et semi-groupes de contractions dans les espaces de Hilbert,
North-Holland, Amsterdam, 1973.




\bibitem{BP} 
 S. Brin; L. Page, The anatomy of a large-scale hypertextual Web search engine,
 Comput. Networks ISDN Syst. 30 (1998), 107-117.\\
 \url{https://doi.org/10.1016/S0169-7552(98)00110-X}.



\bibitem{Ch97} 
 F. Chung, Spectral Graph Theory, American Mathematical Society, Providence, RI, 1997.



\bibitem{Ch07} 
 F. Chung, The heat kernel as the pagerank of a graph, Proc. Nat. Acad. Sci. USA.
 104(50) (2007), 19735--19740.\\
 \url{https://doi.org/10.1073/pnas.0708838104}.



\bibitem{CD80} 
 D.M. Cvetkovi\'{c}; M. Doob; H. Sachs, Spectra of Graphs, Theory and Application, 
 Academic Press, New York, 1980.


\bibitem{FSY}
K. Fujii; T. Soma; Y. Yoshida,
Polynomial-Time Algorithms for Submodular Laplacian Systems,
Theoret. Comput. Sci. 892 (2021) 170--186.\\
\url{https://doi.org/10.1016/j.tcs.2021.09.019}.



\bibitem{F-I-U}
T. Fukao; M. Ikeda; S. Uchida, 
Heat equation on the hypergraph containing vertices with given data,
to appear in J. Math. Anal. Appl. 540(1).\\
\url{https://doi.org/10.1016/j.jmaa.2024.128675}.




\bibitem{IKTU}
M. Ikeda; Y. Kitabeppu; Y. Takai; T. Uehara, 
Coarse Ricci curvature of hypergraphs and its generalization, preprint (2021), arXiv 2102.00698.\\
\url{https://arxiv.org/abs/2102.00698}
  


\bibitem{IMTY}
M. Ikeda; A. Miyauchi; Y. Takai; Y. Yoshida,
Finding cheeger cuts in hypergraphs via heat equation,
Theoret. Comput. Sci. 930, (2022) 1--23.\\
\url{https://doi.org/10.1016/j.tcs.2022.07.006}.




\bibitem{I-U}
M. Ikeda; S. Uchida, 
Nonlinear evolution equation associated with hypergraph Laplacian,
Math. Methods Appl. Sci. 46(8) (2023) 9463-9476.\\
\url{https://doi.org/10.1002/mma.9068}.




\bibitem{Jost}
J. Jost; R. Mulas,
Hypergraph Laplace operators for chemical reaction networks,
Adv. Math. 351 (2019), 870-896. \\
 \url{https://doi.org/10.1016/j.aim.2019.05.025}.




\bibitem{Ken75}
N. Kenmochi, 
Some nonlinear parabolic variational inequalities,
Israel J. Math. 22(3)--(4) (1975), 304--331.\\
\url{https://doi.org/10.1007/BF02761596}.


\bibitem{KM}
Y. Kitabeppu; E. Matsumoto,
Cheng's maximal diameter theorem for hypergraphs, 
Tohoku Math. J. 75(1) (2023), 119--130.

\bibitem{L-M}
P. Li; O. Milenkovic,
Submodular hypergraphs: $p$-Laplacians, Cheeger inequalities and spectral clustering,
Proc. 35th Int. Conf. Mach. Learn. PMLR 80 (2018) 3014-3023.\\
\url{http://proceedings.mlr.press/v80/li18e.html}.



\bibitem{Rock}
R.T. Rockafellar, Convex analysis,
Princeton Math. Ser. 28,
Princeton University Press, Princeton, 1970.



\bibitem{TMIY}
Y. Takai; A. Miyauchi;  M. Ikeda; Y. Yoshida,
Hypergraph Clustering Based on PageRank,
Proc. 26th ACM SIGKDD Int. Conf. Knowledge Discovery {\&} Data Mining (2020) 1970--1978.\\
\url{https://doi.org/10.1145/3394486.3403248}.



\bibitem{Yoshida00}
Y. Yoshida,
Cheeger Inequalities for Submodular Transformations,
Proc. 2019 Annu. ACM-SIAM Sympos.  Discrete Algorithms (SODA) (2019) 2582--2601.\\
\url{https://doi.org/10.1137/1.9781611975482.160}.

 


























\end{thebibliography}
\end{document}